\documentclass[10pt,a4paper]{article}

\usepackage[utf8]{inputenc}
\usepackage[english]{babel}
\usepackage{amsmath}  
\usepackage{amsfonts} 
\usepackage{amsthm}    
\usepackage{amssymb}
\usepackage{graphicx} 
\usepackage{paralist} 
\usepackage{geometry} 
\usepackage{comment}
\usepackage[final]{showlabels}

\usepackage{graphicx,epstopdf,overpic,tikz}   

\author{Christian Bick${}^{1,2,3,4}$ and Tobias Böhle${}^5$ and Christian Kuehn${}^{5,6}$}
\title{Multi-Population Phase Oscillator Networks\\ with Higher-Order Interactions}
\date{{\small
${}^1$Department of Mathematics, University of Exeter, Exeter EX4 4QF, United Kingdom\\
${}^2$Institute for Advanced Study, Technical University of Munich, Lichtenbergstr. 2, 85748 Garching, Germany\\
${}^3$Department of Mathematics, Vrije Universiteit Amsterdam, De Boelelaan 1111, Amsterdam, the Netherlands\\
${}^4$Mathematical Institute, University of Oxford, Oxford OX2 6GG, United Kingdom\\
${}^5$Technical University of Munich, Department of Mathematics (M8), Boltzmannstr.~3, 85748 Garching b.~M\"unchen, Germany\\
${}^6$Complexity Science Hub Vienna, Josefst\"adter Str.~39, 1080 Vienna, Austria}\\
\medskip
\today}

\newcommand{\R}{\mathbb{R}} 
\newcommand{\C}{\mathbb{C}} 
\newcommand{\Z}{\mathbb{Z}} 
\newcommand{\N}{\mathbb{N}} 
\newcommand{\B}{\mathbb{B}} 
\renewcommand{\S}{\mathbb{S}} 

\renewcommand{\d}{\mathrm{d}}

\newcommand{\snorm}[1]{\lvert #1 \rvert_\S}
\newcommand{\norm}[1]{\left\lVert #1 \right\rVert}
\newcommand{\abs}[1]{\left\lvert #1 \right\rvert}

\renewcommand{\Re}{\operatorname{Re}}
\renewcommand{\Im}{\operatorname{Im}}

\theoremstyle{definition}
\newtheorem{definition}{Definition}[section]		
\newtheorem{example}[definition]{Example}			

\theoremstyle{plain} 	
\newtheorem{lemma}[definition]{Lemma}				
\newtheorem{corollary}[definition]{Corollary}		
\newtheorem{theorem}[definition]{Theorem}			
\newtheorem{proposition}[definition]{Proposition}	

\theoremstyle{remark}
\newtheorem{remark}[definition]{Remark}				

\theoremstyle{plain}

\newtheorem{notation}[definition]{Notation}			

\numberwithin{equation}{section}

\begin{document}

\maketitle

\hrule
\paragraph{Abstract.}
The classical Kuramoto model consists of finitely many pairwise coupled oscillators on the circle. In many applications a simple pairwise coupling is not sufficient to describe real-world phenomena as higher-order (or group) interactions take place. Hence, we replace the classical coupling law with a very general coupling function involving higher-order terms. Furthermore, we allow for multiple populations of oscillators interacting with each other through a very general law. In our analysis, we focus on the characteristic system and the mean-field limit of this generalized class of Kuramoto models. While there are several works studying particular aspects of our program, we propose a general framework to work with all three aspects (higher-order, multi-population, and mean-field) simultaneously. Assuming identical oscillators in each population, we derive equations for the evolution of oscillator populations in the mean-field limit. First, we clarify existence and uniqueness of our set of characteristic equations, which are formulated in the space of probability measures together with the bounded-Lipschitz metric. Then, we investigate dynamical properties within the framework of the characteristic system. We identify invariant subspaces and stability of the state, in which all oscillators are synchronized within each population. Even though it turns out that this so called all-synchronized state is never asymptotically stable, under some conditions and with a suitable definition of stability, the all-synchronized state can be proven to be at least locally stable. In summary, our work provides a rigorous mathematical framework upon which the further study of multi-population higher-order coupled particle systems can be based.\vspace{-3mm}
\vspace{-3mm}
\vspace{4mm}

\hrule

\section{Introduction}\label{sec:introduction}

Interacting oscillatory processes are abundant in science and technology, whether it is pacemaker cells in the heart~\cite{Strogatz1993}, neural networks in the brain~\cite{Walker1969}, the synchronization of chemical oscillators~\cite{Kiss2002}, Josephson junctions~\cite{Wiesenfeld1995}, synchronous flashing fireflies~\cite{Buck1968}, stock prices in financial markets~\cite{DalMasoPeron2011}, and even a group of people causing a bridge to swing by marching over it in step~\cite{Strogatz2005}. From a mathematical perspective, many such systems can be described as networks of coupled phase oscillators: While each node in the network is a simple oscillatory process whose state is given by a single phase-like variable on the circle, the network dynamics---the collective dynamics of all nodes---can be quite intricate.
Probably the most prominent example of collective network dynamics is synchrony when all nodes evolve in unison~\cite{Strogatz2004}. Synchrony can come in many forms, whether oscillators synchronize in phase or in frequency. Synchrony may be global across the network or may be localized in part of the network to give rise to synchrony patterns. Synchrony in its many forms often is also relevant for the function of a network dynamical system: In neural networks, synchrony is linked to cognitive function as well as disease~\cite{Uhlhaas2006}.

A classical question in network dynamical systems is how the network structure and functional interactions shape the collective network dynamics.
The Kuramoto model has been instrumental in understanding synchronization in coupled oscillators~\cite{Kuramoto,Strogatz2000,Acebron2005}. In the case of~$N$ identical oscillators, the state $\theta_k\in\S:=\R/2\pi\Z$ of oscillator~$k\in\{ 1,\dotsc,N\}$ evolves according to
\begin{equation}
    \label{eq:Kuramoto}
    \dot\theta_k = \omega + \frac{1}{N}\sum_{j=1}^N\sin(\theta_j-\theta_k),
\end{equation}
where $\omega\in\R$ is the intrinsic frequency of the oscillators. To understand (global) synchronization, the Kuramoto model makes several simplifying assumptions: The network is homogeneous and interactions are pairwise and all-to-all, mediated by the sine of the phase differences. The dynamics of~\eqref{eq:Kuramoto} are particularly simple: The dynamics are effectively two-dimensional and for generic initial conditions all oscillators synchronize~\cite{Watanabe1994}.

However, many real-world network dynamical systems are more complex and cannot be captured using the Kuramoto model. 
First, one needs to consider interaction functions that contain more than a single harmonic~\cite{Daido1994}, for example if the interactions are state-dependent~\cite{Ashwin2019}. Considering a general coupling function $g:\S\to\R$ breaks the degeneracy and chaotic dynamics are possible even for small oscillators~\cite{Bick2011}. 
Second, many networks that arise in real-world systems are not homogeneous or all-to-all coupled but have some form of modularity or community structure~\cite{Girvan2002,Newman2006}: There are different communities or populations that are characterized by the property that the coupling within a population is different from the coupling between populations. Even if the populations are identical, multi-population oscillator networks give rise to a variety of synchrony patterns; see~\cite{Bick2018c} for a recent review.
Third, classical network dynamical systems, such as the Kuramoto model~\eqref{eq:Kuramoto}, assume that interactions take place in terms of pairs: The influence of two nodes on a third is simply the sum of two individual contributions.
For example, Skardal and Arenas~\cite{SkardalArenas} consider the phase oscillator dynamics given by
\begin{align}\label{eq:System_Skardal}
	\dot\theta_i = \omega + \frac{K_1}{N}\sum_{j=1}^{N}\sin(\theta_j-\theta_i) + \frac{K_2}{N^2}\sum_{j=1}^{N}\sum_{l=1}^{N}\sin(2\theta_j-\theta_l-\theta_i) + \frac{K_3}{N^3}\sum_{j=1}^{N}\sum_{l=1}^{N}\sum_{m=1}^{N}\sin(\theta_j-\theta_l+\theta_m-\theta_i)
\end{align}
with nonadditive interactions that involve triplets and quadruplets of oscillators. Indeed, nonpairwise higher-order interactions arise naturally when considering phase reductions (cf.~\cite{Ashwin2016,Leon2019a} and the appendix) and recent work has highlighted the dynamical importance of higher-order interactions~\cite{Battistonetal1,Bick2021}.
In all these cases, going beyond the Kuramoto model leads to new dynamics. Further examples of new dynamics are possible if these features are combined, e.g., modular networks of phase oscillator networks with higher-order interactions allow for heteroclinic structures between different synchrony patterns~\cite{Bick2017c,Bick2019a,Bick2019}.

So far, rigorous insights that relate network structure and dynamics have relied on specific assumptions on the size of the network and the type of interactions. Classical dynamical systems techniques can be applied if the networks have relatively few nodes: The heteroclinic structures between synchrony patterns have been shown to exist in networks of up to nine phase oscillators~\cite{Bick2019a}. Since many real-world networks have many nodes, it is often instructive to consider these networks in the limit of infinitely many nodes. In this limit, the Ott--Antonsen reduction~\cite{Ott2008,Bick2018c} has been instrumental to elucidate the dynamics of coupled oscillator populations. The main drawback is that this reduction requires network interactions that are mediated by a single harmonic, whether it is pairwise or contains higher-order terms. Moreover, the limit itself has only been rigorously justified in a restricted setup.

Here, we give general results on the mean-field limit of multi-population phase oscillator networks with higher-order interactions. We rigorously describe the evolution of the mean-field limit where each population is represented by a probability measure that describes the distribution of oscillators on the circle. Moreover, we identify how phase space is organized for a class of coupled oscillator networks with higher-order interactions and calculate the (local) stability properties of synchrony patterns. Our results contribute to the understanding of coupled oscillator networks in several ways. First, we generalize the results of~\cite{Lancellotti2005} to transport equations involving finitely many oscillator populations with higher-order interactions. Second, our stability analysis complements the work in~\cite{Carrillo2014} and~\cite{Dietert2018} by considering a general setting with multiple oscillator populations and general coupling with nonsinusoidal pairwise and nonpairwise higher-order interactions. Third, our explicit stability results on a measure-valued evolution shows in this setting one cannot expect asymptotic stability of full phase synchrony but just a weaker form of stability. Finally, our results provide a first step towards understanding global dynamical phenomena in the general mean-field limit: Our stability results outline necessary conditions to prove that the heteroclinic structures for small networks exist in the mean-field limit of large networks.

This work is organized as follows. In the remainder of this section we fix some notation that will be used throughout and introduce a general system of equations that describe the network dynamics. In Section~\ref{sec:solution_theory} we establish the existence and uniqueness of the equations in the space of probability measures. This section also covers the relation of the mean-field limit to the continuity equation and networks of finitely many oscillators. 
The main results regarding synchronization are given in Section~\ref{sec:dynamics}. 
We apply our results to three explicit examples of coupled phase oscillator networks in Section~\ref{sec:Examples}. Finally, in Section~\ref{sec:conclusion} we give some concluding remarks and an outlook on future research.

\section*{Notation}

We first fix some notation that will be used throughout this paper. Let~$\mathcal P(X)$ denote the set of all Borel probability measures on the set~$X$. If $\S=\R/2\pi\Z$ is the unit circle then~$\mathcal P(\S)$ represents the set of all Borel probability measures on the circle. The symbol $\mathcal P_\mathrm{ac}(\S)$ is used to denote the set of absolutely continuous probability measures, i.e., those which have a density, on the unit circle. Whenever we write $\alpha_1-\alpha_2$ for two points $\alpha_1, \alpha_2\in \S$ on the circle, we refer to the value of $\alpha_1-\alpha_2 \in [0,2\pi)$. Further, let us define open intervals on the circle as
\begin{align*}
    (\alpha_1,\alpha_2) := \{ \alpha_1 + t (\alpha_2-\alpha_1) : t\in (0,1)\}.
\end{align*}
Assuming $\alpha_1$ and $\alpha_2$ are represented by values $\alpha_1, \alpha_2\in (-\pi, \pi]$ we use the notation
\begin{align*}
	|\alpha_1-\alpha_2|_\S := \min(\alpha_1-\alpha_2, 2\pi-(\alpha_1-\alpha_2)).
\end{align*}
To compare two measures $\mu, \nu \in \mathcal P(\S)$, we use the Wasserstein-$1$ distance \cite{Villani2003}, which is also referred to as the bounded-Lipschitz distance
\begin{subequations}
\label{eq:WassersteinMetric}
\begin{align}
	\label{eq:WassersteinCoupling}
	W_1(\mu,\nu) &:=  \mathop{\inf_{\gamma\in\mathcal P(\S\times\S)}}_{M_1\gamma=\mu,\ M_2\gamma = \nu} \int_{\S\times\S} |\alpha-\beta|_\S\  \gamma(\d\alpha, \d\beta)\\
	\label{eq:WassersteinLipschitz}
	& = \sup_{f \in \mathcal{D}} \left\lvert\int_\S f(\alpha) \ \d\mu(\alpha) - \int_\S f(\alpha)\ \d\nu(\alpha)\right\lvert,
\end{align}
\end{subequations}
where $M_1\gamma$ and $M_2\gamma$ are the marginals of $\gamma$, i.e., the push-forward measures under the map $(\alpha, \beta)\mapsto \alpha$ and $(\alpha, \beta)\mapsto \beta$ and
\begin{align*}
	\mathcal{D} := \{ f\in C(\S): \lvert f(\alpha)-f(\beta)\rvert\le\snorm{\alpha-\beta} \text{ for all } \alpha,\beta\in\S\}.
\end{align*}
Further, for $n\in\N$ we write $[n] := \{1, \dotsc, n\}$ and for $R\in\N$ we define the multi-index $s = (s_1,\dots,s_R)\in [M]^R$. Then, given $\mu = (\mu_1,\dots,\mu_M)\in \mathcal P(\S)^M$, we define the measure
\begin{align*}
    \mu^{(s)} = (\mu_{(s_1)}, \dots,\mu_{(s_R)})
\end{align*}
and write $\abs{s} = R$, $\bar s = \max_{j\in [M]} \abs{\{i : s_i = j\}}$.

\section{Solution Theory}\label{sec:solution_theory}

In this section we first introduce the system of characteristic equations stated below as \eqref{eq:SystemOfEquations}--\eqref{eq:coupling}. Then, we establish its well-posedness, its continuous dependence on initial conditions, and finally relate the solution of the characteristic system to mean-field limit equations of Vlasov--Fokker--Planck type and finite-dimensional generalized Kuramoto models with multiple populations and higher-order coupling.

\subsection{The System of Characteristic Equations}

In this paper, we consider the dynamics of~$M\in\N$ coupled phase oscillator populations. We now introduce a general set of equations that describes the network evolution, where the state of population $\sigma\in[M]$ is given by a probability measure~$\mu_\sigma$. 

The network interactions are determined by a multi-index $s^\sigma\in[M]^{R_\sigma}$ for each population together with Lipschitz continuous coupling functions $G_\sigma\colon \S^{\abs{s^\sigma}}\times\S\to \R$. Specifically, these coupling functions are supposed to be $L$-Lipschitz when $\S^{\abs{s^\sigma}}\times\S$ is considered with the metric $d(\alpha,\beta) = \sum_{i=1}^{\abs{s^\sigma}+1} \snorm{\alpha_i-\beta_i}$. If $\mu^\mathrm{in} = (\mu^\mathrm{in}_1,\dots,\mu^\mathrm{in}_M)\in \mathcal P(\S)^M$ denotes the initial state of the network,~$\#$ denotes the push-forward operator and $\mu = (\mu_1,\dots,\mu_M)$, then the evolution of $\mu(t) = (\mu_1(t), \dots,\mu_M(t))$ is determined by the characteristic equations
\begin{subequations}
\label{eq:SystemOfEquations}
\begin{align}
	\label{eq:SystemOfEquations_Derivative}
	\partial_t \Phi_\sigma(t,\xi^\mathrm{in}_\sigma, \mu^\mathrm{in}) &= (\mathcal K_\sigma\mu(t))(\Phi_\sigma(t,\xi^\mathrm{in}_\sigma, \mu^\mathrm{in}))\\
	\label{eq:SystemOfEquations_PushForward}
	\mu_\sigma(t) &= \Phi_\sigma(t,\cdot, \mu^\mathrm{in})\#\mu^\mathrm{in}_\sigma\\
	\label{eq:SystemOfEquations_InitialCond}
	\Phi_\sigma(0,\xi^\mathrm{in}_\sigma,\mu^\mathrm{in})&= \xi^\mathrm{in}_\sigma.
\end{align}
\end{subequations}
for $\sigma\in[M]$ and the evolution operator
\begin{align}
\label{eq:coupling}
		(\mathcal K_\sigma\mu)(\phi) = \omega_\sigma + \int_{\S^{\abs{s^\sigma}}} G_\sigma(\alpha,\phi)\ \d\mu^{(s^\sigma)}(\alpha),
\end{align}
where $\omega_\sigma\in\R$ is the instantaneous frequency of all oscillators in population~$\sigma$. 
We remark that the general idea of using a mean-field formulation involving probability measures instead of densities is quite classical~\cite{Golse2013}. Our equations~\eqref{eq:SystemOfEquations}--\eqref{eq:coupling} provide a very general variant of this principle allowing for multi-population higher-order coupled systems. Yet, the formulation also naturally reduces to classical cases, as we illustrate in Section \ref{sec:Examples}.

\begin{definition}\label{def:solution}
    Let $\mu^\mathrm{in} = (\mu^\mathrm{in}_1,\dots,\mu^\mathrm{in}_M)\in \mathcal P(\S)^M$. If functions $t\mapsto \Phi_\sigma(t,\xi^\mathrm{in}_\sigma,\mu^\mathrm{in})$ solve the ODE \eqref{eq:SystemOfEquations_Derivative} together with \eqref{eq:SystemOfEquations_PushForward},\eqref{eq:SystemOfEquations_InitialCond} and \eqref{eq:coupling}, they are referred to as the \emph{mean-field characteristic flow}. In this case, $\mu\in C_{\mathcal P(\S)}^M$, given by \eqref{eq:SystemOfEquations_PushForward}, is a \emph{solution} of the system \eqref{eq:SystemOfEquations}--\eqref{eq:coupling}.
\end{definition}

\begin{remark}
	For a given mean-field characteristic flow $\Phi_\sigma$, the solution $\mu$ of the characteristic system is uniquely given by \eqref{eq:SystemOfEquations_PushForward}. Conversely, for a given solution $\mu\in C_{\mathcal P(\S)}^M$, the mean-field characteristic flow~$\Phi_\sigma$ is unique, as one can see by integrating \eqref{eq:SystemOfEquations_Derivative} and \eqref{eq:SystemOfEquations_InitialCond}.
\end{remark}

\subsection{Existence, Uniqueness and Continuous Dependence on Initial Conditions for the Characteristic System}\label{sec:existence_uniqueness_contdepini}

We start with existence and uniqueness building upon ideas by Neunzert~\cite{Neunzert1978} developed in the context of more classical single-population kinetic models, which lead to similar mean-field limits in comparison to our system \eqref{eq:SystemOfEquations}--\eqref{eq:coupling}.
 Since the proof given in~\cite{Neunzert1978} makes abstract assumptions about the coupling function and only deals with one population, there are quite some differences to the existence and uniqueness proof of our system, which is why we decided to include the full details here. First, we have to define a suitable space for solutions.

\begin{definition}
	Let $T>0$. A function $\mu\colon[0,T]\to \mathcal P(\S)$ is \emph{weakly continuous} if for all $f\in C(\S)$ the map
	\begin{align*}
		t\mapsto \int_\S f(\phi) \ \mu(t,\d \phi)
	\end{align*}
	is continuous. Let $C_{\mathcal P(\S)}$ be the set of all weakly continuous functions $\mu\colon [0,T] \to \mathcal P(\S)$.
\end{definition}

\begin{remark}
    As $T>0$ is arbitrary, we do not explicitly include $T$ in the notation $C_{\mathcal P(\S)}$ for functions mapping from $[0,T]$ to $\mathcal P(\S)$. Since the following existence and uniqueness result as well as results regarding continuous dependence on initial conditions are valid for all $T>0$, they can be extended to hold on the half-open interval $[0,\infty)$.
\end{remark}

\begin{lemma}\label{lem:basicmetricspace}
	The space $C_{\mathcal P(\S)}$ together with the metric
	\begin{align*}
		d(\mu,\nu) := \sup_{t\in[0,T]} W_1(\mu(t),\nu(t))
	\end{align*}
	is a complete metric space.
\end{lemma}
\begin{proof}
	Let $(\gamma_n)_{n\in\N}$ be a Cauchy sequence in $\mathcal P(\S)$. Then, we obtain from~\cite[Satz 3]{Kellerer1972} the convergence of $(\gamma_n)_{n\in\N}$ to an arbitrary positive measure. Testing with the constant $1$-function, i.e., choosing $f\equiv 1$ in the representation \eqref{eq:WassersteinLipschitz}, yields that the Cauchy sequence is even converging to a probability measure. This shows the completeness of $(\mathcal P(\S), W_1)$.\\
	Now, let $(\mu_n)_{n\in\N}$ be a Cauchy sequence in $(C_{\mathcal P(\S)}, d)$. The completeness of $(\mathcal P(\S), W_1)$ causes the existence of measure-valued functions $\mu_\infty\colon [0,T]\to \mathcal P(\S)$. To show weak continuity of $\mu_\infty$, let $f\in C(\S)$ first be $1$-Lipschitz continuous and calculate
	\begin{align*}
	    \left| \int_\S f(\phi)\ \mu_n(t,\d\phi) - \int f(\phi)\ \mu_\infty(t,\d\phi)\right| \le W_1(\mu_n(t), \mu_\infty(t)) \to 0
	\end{align*}
	uniformly in $t$ as $n\to \infty$. Continuity of $\int_\S f(\phi) \ \mu_n(t,\d\phi)$ in $t$ for each $n\in \N$ therefore implies continuity of $\int f(\phi)\ \mu_\infty(t,\d\phi)$ in $t$. Finally, Portemanteau's Theorem (see eg. \cite[Theorem 13.17]{Klenke2008}) allows us to lift continuity of $\int f(\phi)\ \mu_\infty(t,\d\phi)$ from all $1$-Lipschitz continuous functions $f$ to all $f\in C(\S)$.
\end{proof}

\begin{lemma}
For every $\alpha \in \R$, the space $C_{\mathcal P(\S)}^M$ together with the metric
\begin{align}\label{eq:metric}
	d_\alpha(\mu, \nu) := \sup_{t\in [0,T]} e^{-\alpha t} \sum_{\sigma = 1}^M W_1(\mu_\sigma(t), \nu_\sigma(t))
\end{align}
is a complete metric space.
\end{lemma}
\begin{proof}
	Follows from Lemma \ref{lem:basicmetricspace}.
\end{proof}

For a given $\mu\in C_{\mathcal P(\S)}^M$, let~$T_{t,s}^\sigma[\mu]$ denote the flow induced by the velocity field $(\mathcal K_\sigma \mu(t))(\phi)$, i.e., 
\begin{align}\label{eq:Tflow}
	\frac\d{\d t} T_{t,s}^\sigma[\mu]\phi = (\mathcal K_\sigma\mu(t))(T_{t,s}^\sigma[\mu]\phi),\quad T_{s,s}^\sigma[\mu]\phi = \phi.
\end{align}
We proceed to verify several assumptions and constructions in the arguments in~\cite{Neunzert1978} within our new setting. 

\begin{lemma}[cf.~\cite{Neunzert1978}, Assumption on page 236]\label{lem:LipschitzVelocity} For all $\sigma \in [M]$, $\phi,\psi\in\S$ and $\mu\in \mathcal P(\S)^M$ we have
	\begin{align*}
		|(\mathcal K_\sigma \mu)(\phi)-(\mathcal K_\sigma \mu)(\psi)| \le L \snorm{\phi-\psi}.
	\end{align*}
\end{lemma}

\begin{proof}
    Using the assumption about Lipschitz continuity of $G_\sigma$, we can estimate
	\begin{align*}
		|(\mathcal K_\sigma\mu)(\phi)-(\mathcal K_\sigma\mu)(\psi)| &= \left| \int_{\S^{\abs{s^\sigma}}} G_\sigma(\alpha, \phi) - G_\sigma(\alpha,\psi) \ \d\mu^{(s^\sigma)}(\alpha)\right| \\
		&\le \int_{\S^{\abs{s^\sigma}}} L \snorm{\phi-\psi} \ \d\mu^{(s^\sigma)}(\alpha)\\
		&=L\snorm{\phi-\psi}.
	\end{align*}
	This completes the proof.
\end{proof}

\begin{corollary}\label{cor:LipschitzFlow}
	Given $\mu\in C_{\mathcal P(\S)}^M$, the flow $T^\sigma_{t,0}[\mu]$ is Lipschitz continuous for all $\sigma \in [M]$. In particular,
	\begin{align*}
		\snorm{T_{t,0}^\sigma[\mu]\phi - T^\sigma_{t,0}[\mu]\psi} \le e^{Lt}\snorm{\phi-\psi}.
	\end{align*}
\end{corollary}

\begin{proof}
	This is a direct consequence of Lemma \ref{lem:LipschitzVelocity} and Gronwall's Lemma.
\end{proof}

The following Lemma provides an important tool when dealing with multiple populations and higher-order coupling:
\begin{lemma}\label{lem:HighDimWasserstein}
	Let $n\in\N, \mu_1,\dots,\mu_n,\nu_1,\dots,\nu_n\in \mathcal P(\S)$, $\mu := \mu_1\otimes\cdots\otimes\mu_n\in \mathcal P(\S^n)$, $\nu := \nu_1\otimes\cdots\otimes\nu_n\in \mathcal P(\S^n)$ and $g\colon \S^n\to \R$ be an $L$-Lipschitz continuous function with respect to the metric $d(\alpha,\beta) = \sum_{k=1}^{n}\snorm{\alpha_k-\beta_k}$. Then,
	\begin{align*}
		\left|\int_{\S^n} g(\alpha)\ \d\mu(\alpha) - \int_{\S^n} g(\beta) \ \d\nu(\beta)\right|\le L \sum_{i=1}^{n}W_1(\mu_i,\nu_i).
	\end{align*}
\end{lemma}

\begin{proof}
	In this proof we work with the Wasserstein-$1$ distance in $\mathcal P(\S^n)$ and its dual representation \cite{Villani2003}. For two measures $\mu,\nu\in \mathcal P(\S^n)$ they are given by
	\begin{align}
		\nonumber
		W_1(\mu,\nu) &= \mathop{\inf_{\pi\in \mathcal P(\S^n\times \S^n)}}_{M_{(1,\dots,n)}\pi = \mu, M_{(n+1,\dots,2n)}\pi = \nu} \int_{\S^n\times \S^n} d(\alpha,\beta)\ \d\pi(\alpha,\beta)\\
		\label{eq:HighDimWassersteinDual}
		&=\mathop{\sup_{f\in C(\S^n)}}_{|f(\alpha)-f(\beta)|<d(\alpha,\beta)} \left| \int_{\S^n} f(\alpha)\ \d\mu(\alpha) - \int_{\S^n} f(\beta)\ \d\nu(\beta)\right|,
	\end{align}
	where $M_{(1,\dots,n)}\pi$ is the push-forward measure of $\pi$ under the map $(\alpha_1,\dots,\alpha_n,\beta_1,\dots,\beta_n)\mapsto (\alpha_1,\dots,\alpha_n)$ and $M_{(n+1,\dots,2n)}\pi$ is the push-forward measure of $\pi$ under the map $(\alpha_1,\dots,\alpha_n,\beta_1,\dots,\beta_n)\mapsto (\beta_1,\dots,\beta_n)$.
	Let us denote
	\begin{align*}
			\mathcal D_1 & := \{ \pi \in \mathcal P(\S^n\times\S^n): M_{(1,\dots,n)}\pi = \mu, M_{(n+1,\dots,2n)}\pi = \nu_1\otimes \mu_2\otimes\dots\otimes\mu_n \},  \\
			\mathcal D_2 & := \{ \gamma\in \mathcal P(\S\times\S): M_1\gamma=\mu_1,M_2\gamma = \nu_1\},\\
			\mathcal D_3 & := \{ \pi \in \mathcal P(\S^n\times\S^n):\exists \gamma\in \mathcal D_2:\\
			&\qquad   \d\pi(\alpha_1,\dots,\alpha_n,\beta_1,\dots,\beta_n) = \d\gamma(\alpha_1,\beta_1)\d\delta_{\{\alpha_2=\beta_2\}}(\alpha_2)\d\mu_2(\beta_2)\cdots\d\delta_{\{\alpha_n=\beta_n\}}(\alpha_n)\d\mu_n(\beta_n) \}.
	\end{align*}
	Note that $\mathcal D_3\subset \mathcal D_1$, which is why
	\begin{align*}
		&W_1(\mu_1\otimes\dots\otimes\mu_n, \nu_1\otimes\mu_2\otimes\dots\otimes\mu_n)\\
		&\qquad\qquad= \inf_{\pi\in \mathcal D_1} \int_{\S^n\times\S^n} d(\alpha,\beta)\ \d\pi(\alpha,\beta)\\
		&\qquad\qquad\le \inf_{\pi\in \mathcal D_3} \int_{\S^n\times\S^n} d(\alpha,\beta)\ \d\pi(\alpha,\beta)\\
		&\qquad\qquad= \inf_{\gamma\in \mathcal D_2} \int_{\S^n\times\S^n} d(\alpha,\beta) \ \d\gamma(\alpha_1,\beta_1)\d\delta_{\{\alpha_2=\beta_2\}}(\alpha_2)\d\mu_2(\beta_2)\cdots\d\delta_{\{\alpha_n=\beta_n\}}(\alpha_n)\d\mu_n(\beta_n)\\
		&\qquad\qquad= \inf_{\gamma\in \mathcal D_2} \int_{\S\times\S} \snorm{\alpha_1-\beta_1}\ \d\gamma(\alpha_1,\beta_1)\\
		&\qquad\qquad= W_1(\mu_1,\nu_1).
	\end{align*}
	Using \eqref{eq:HighDimWassersteinDual} and the above calculation we can finally calculate
	\begin{align*}
		\left|\int_{\S^n} g(\alpha)\ \d\mu(\alpha) - \int_{\S^n} g(\beta) \ \d\nu(\beta)\right| &\le L W_1(\mu,\nu)\\
		&= L W_1(\mu_1\otimes\cdots\otimes\mu_n, \nu_1\otimes\cdots\otimes\nu_n)\\
		&\le L W_1(\mu_1\otimes\cdots\otimes\mu_n, \nu_1\otimes\mu_2\otimes\cdots\otimes\mu_n)\\
		&\qquad + L W_1(\nu_1\otimes\mu_2\otimes\cdots\otimes\mu_n, \nu_1\otimes\nu_2\otimes\mu_3\otimes\cdots\otimes\mu_n)\\
		&\qquad + \dots\\
		&\qquad + L W_1(\nu_1\otimes\cdots\otimes \nu_{n-1}\otimes \mu_n, \nu_1\otimes \cdots\otimes\nu_n)\\
		&\le L \sum_{i=1}^{n}W_1(\mu_i,\nu_i).
	\end{align*}
\end{proof}

\begin{lemma}[cf. \cite{Neunzert1978}, Assumption on page 237]\label{lem:KMeasure} For all $\sigma \in [M]$, $\phi\in\S$ and $\mu,\nu\in \mathcal P(\S)^M$ we have
	\begin{align*}
		|(\mathcal K_\sigma \mu)(\phi)-(\mathcal K_\sigma\nu)(\phi)| \le L\bar{s^\sigma} \sum_{i = 1}^M W_1(\mu_i,\nu_i).
	\end{align*}
\end{lemma}

\begin{proof}
	Using Lemma \ref{lem:HighDimWasserstein}, we can estimate
	\begin{align*}
		|(\mathcal K_\sigma \mu)(\phi)-(\mathcal K_\sigma\nu)(\phi)| &\le  \left| \int_{\S^{\abs{s^\sigma}}} G_\sigma(\alpha,\phi) \ \d\mu^{(s^\sigma)}(\alpha) -  \int_{\S^{\abs{s^\sigma}}} G_\sigma(\beta,\phi)\ \d\nu^{(s^\sigma)}(\beta) \right|\\
		&\le L \sum_{i=1}^{\abs{s^\sigma}} W_1(\mu_{(s^\sigma_i)},\nu_{(s^\sigma_i)}) \le L\bar{s^\sigma}\sum_{i=1}^M W_1(\mu_i,\nu_i),
	\end{align*}
	which completes the proof.
\end{proof}

\begin{lemma}\label{lem:estimate} For all $\mu,\nu\in C_{\mathcal P(\S)}^M$ and
	\begin{align}\label{eq:w_def}
		w(t):= \sum_{i=1}^M W_1(\mu_i(t), \nu_i(t))
	\end{align}
we have
	\begin{align*}
		&|T_{t,0}^\sigma[\mu]\phi-T^\sigma_{t,0}[\nu]\phi|_\S \le L\bar{s^\sigma} e^{Lt} \int_0^t w(\tau) e^{-\tau L}\ \d\tau.
	\end{align*}
\end{lemma}

\begin{proof}
Integrating \eqref{eq:Tflow} and then using Lemmas \ref{lem:LipschitzVelocity} and \ref{lem:KMeasure}, we can estimate
	\begin{align*}
		|T_{t,0}^\sigma[\mu]\phi-T^\sigma_{t,0}[\nu]\phi|_\S &= \left\lvert\int_0^t (\mathcal K_\sigma\mu(\tau))(T_{\tau,0}^\sigma[\mu]\phi)-(\mathcal K_\sigma\nu(\tau))(T_{\tau,0}^\sigma[\nu]\phi) \ \d\tau\right\rvert\\
		&\le \left\lvert\int_0^t (\mathcal K_\sigma\mu(\tau))(T_{\tau,0}^\sigma[\mu]\phi)-(\mathcal K_\sigma\mu(\tau))(T_{\tau,0}^\sigma[\nu]\phi) \ \d\tau\right\rvert\\
		&\qquad+\left\lvert\int_0^t (\mathcal K_\sigma\mu(\tau))(T_{\tau,0}^\sigma[\nu]\phi)-(\mathcal K_\sigma\nu(\tau))(T_{\tau,0}^\sigma[\nu]\phi) \ \d\tau\right\rvert\\
		&\le \int_0^t L\snorm{T_{\tau,0}^\sigma[\mu]\phi - T_{\tau,0}^\sigma[\mu]\phi} \ \d\tau\\
		&\qquad + \int_0^t L \bar{s^\sigma}\underbrace{\sum_{i=1}^M W_1(\mu_i(\tau),\nu_i(\tau))}_{=w(\tau)}\ \d\tau.
	\end{align*}
	Now introduce
	\begin{align*}
		v_\sigma(t) := \snorm{T_{t,0}^\sigma[\mu]\phi-T_{t,0}^\sigma[\nu]\phi}.
	\end{align*}
	Then, we have
	\begin{align*}
		v_\sigma(t)\le L\int_0^t v_\sigma(\tau) \ \d\tau + L\bar{s^\sigma} \int_0^t w(\tau)\ \d\tau.
	\end{align*}
	Gronwall's Lemma gives us
	\begin{align*}
		v_\sigma(t)\le L\bar{s^\sigma} e^{Lt} \int_0^t w(\tau) e^{-\tau L}\ \d\tau,
	\end{align*}
	which was the claim.
\end{proof}

\begin{lemma}[{\cite[Lemma 1]{Neunzert1978}}]\label{lem:Wasserstein_sup}
    Let $h_1,h_2\colon \S\to \S$ be bijective measurable mappings and let $\mu\in \mathcal P(\S)$. Then,
    \begin{align*}
        W_1(h_1\# \mu, h_2\#\mu) \le \sup_{\phi\in\S}\snorm{h_1(\phi)-h_2(\phi)}.
    \end{align*}
\end{lemma}

For given $\mu^\mathrm{in}\in \mathcal P(\S)^M$ we now consider the mapping 
\begin{align}\label{eq:ContractionMapping}
	A\colon C_{\mathcal P(\S)}^M \to C_{\mathcal P(\S)}^M \quad \text{with} \quad (A\mu)_\sigma(t) := T_{t,0}^\sigma[\mu] \# \mu^\mathrm{in}_\sigma.
\end{align}

\begin{lemma} The mapping $A$, defined in \eqref{eq:ContractionMapping}, is indeed a self mapping.
\end{lemma}
\begin{proof}
	It is clear, that $(A\mu)_\sigma(t)$ is again a probability measure, so it is left to show that $t\mapsto(A\mu)_\sigma(t)$ is weakly continuous. Note that for each $\phi \in \S$, the map $t\mapsto T^\sigma_{t,0}[\mu]\phi$ is continuous (even differentiable). Thus, for any $f\in C(\S)$, the composition $t\mapsto f(T^\sigma_{t,0}[\mu]\phi)$ is continuous as well and uniformly bounded. Consequently, by the change of variables rule for push forward measures and the dominated convergence theorem, we have
	\begin{align*}
		\lim_{t\to t^\star} \int_\S f(\phi)\ (A\mu)_\sigma(t,\d\phi) &= \lim_{t\to t^\star} \int_\S f(T^\sigma_{t,0}[\mu]\phi)\ \mu_\sigma^\mathrm{in}(\d\phi) = \int_\S \lim_{t\to t^\star}f(T^\sigma_{t,0}[\mu]\phi)\ \mu_\sigma^\mathrm{in}(\d\phi)\\
		&= \int_\S f(T^\sigma_{t^\star,0}[\mu]\phi)\ \mu_\sigma^\mathrm{in}(\d\phi) = \int_S f(\phi)\ (A\mu)_\sigma(t^\star,\d\phi),
	\end{align*}
	which proves weak continuity.
\end{proof}

Finally, we can establish existence and uniqueness:

\begin{theorem}[cf.~\cite{Neunzert1978}, Theorem 2]\label{thm:ExistenceMeasure}
	For given $\mu^\mathrm{in} = (\mu_1^\mathrm{in},\dots,\mu_M^\mathrm{in})\in \mathcal P(\S)^M$ there exists a unique solution $\mu\in C_{\mathcal P(\S)}^M$ for the system \eqref{eq:SystemOfEquations}--\eqref{eq:coupling}.
\end{theorem}

\begin{proof}
	Let us prove, that for $\alpha$ large enough, $A$ is a contraction with respect to the metric $d_\alpha$, defined in~\eqref{eq:metric}. To estimate $d_\alpha(A\mu, A\nu)$, we use Lemmas~\ref{lem:estimate} and~\ref{lem:Wasserstein_sup}:
	{\allowdisplaybreaks
	\begin{align*}
		d_\alpha(A\mu,A\nu) &= \sup_{t\in [0,T]} e^{-\alpha t} \sum_{\sigma=1}^{M}W_1((A\mu)_\sigma(t), (A\nu)_\sigma(t))\\
		&= \sup_{t\in [0,T]} e^{-\alpha t} \sum_{\sigma=1}^{M}W_1(T_{t,0}^\sigma [\mu] \#\mu_\sigma^\mathrm{in},  T_{t,0}^\sigma [\nu]\#\mu^\mathrm{in}_\sigma)\\
		&\le \sup_{t\in [0,T]} e^{-\alpha t} \sum_{\sigma=1}^{M}\sup_{\phi\in \S} \snorm{T_{t,0}^\sigma[\mu]\phi - T_{t,0}^\sigma[\nu]\phi}\\
		&\le \sup_{t\in [0,T]} e^{-\alpha t} \sum_{\sigma=1}^{M} L\bar{s^\sigma} e^{Lt}\int_0^t w(\tau) e^{-\tau L} \ \d\tau\\
		&\le L\left(\sum_{\sigma=1}^{M}\bar{s^\sigma}\right) \sup_{t\in [0,T]} e^{-\alpha t+Lt}\int_0^t e^{-\tau (L-\alpha)} w(\tau) e^{-\tau\alpha} \ \d\tau\\
		&\le L\left(\sum_{\sigma=1}^{M}\bar{s^\sigma}\right) \sup_{t\in [0,T]} e^{-\alpha t+Lt}\int_0^t e^{-\tau (L-\alpha)} \left(\sup_{\tau\in [0,T]} w(\tau) e^{-\tau\alpha}\right) \ \d\tau\\
		&\le L\left(\sum_{\sigma=1}^{M}\bar{s^\sigma}\right) \sup_{t\in [0,T]} e^{-\alpha t+Lt}\int_0^t e^{-\tau (L-\alpha)} d_\alpha(\mu,\nu) \ \d\tau\\
		&\le L\left(\sum_{\sigma=1}^{M}\bar{s^\sigma}\right)d_\alpha(\mu,\nu) \sup_{t\in [0,T]} e^{-\alpha t+Lt} \frac{e^{-t(L-\alpha)}-1}{\alpha-L}\\
		&\le \left(\sum_{\sigma=1}^{M}\bar{s^\sigma}\right) \frac{L d_\alpha(\mu,\nu)}{\alpha-L } \underbrace{\sup_{t\in [0,T]} \left(1-e^{-\alpha t+Lt}\right)}_{\le 1 \text{ for } \alpha > L}.\\
	\end{align*}
	}%
	Now, choose $\alpha > L$ such that
	\begin{align*}
		\frac{\left(\sum_{\sigma=1}^{M}\bar{s^\sigma}\right)L}{\alpha-L} <1.
	\end{align*}
	This shows that $A$ is a contraction for suitable values of $\alpha$. Therefore, by the contraction mapping principle, $A$ has a unique fixed point $\mu$. To conclude, if $\mu$ is the fixed point of $A$, $\mu$ satisfies \eqref{eq:SystemOfEquations}--\eqref{eq:coupling} with $\Phi_\sigma(t,\xi^\mathrm{in}_\sigma,\mu^\mathrm{in}) = T_{t,0}^\sigma[\mu]\xi^\mathrm{in}_\sigma$. Conversely, if $\mu$ satisfies \eqref{eq:SystemOfEquations}--\eqref{eq:coupling}, $\mu$ is also a fixed point of $A$, which completes the proof.
\end{proof}

As a next step, we want to establish continuous dependence on initial conditions for the characteristic system. Since this system links to the finite-dimensional generalized Kuramoto models via the empirical measure as well as to the mean-field limit, we may later conclude from a suitable continuous dependence that the mean-field limit is a good approximation on finite time scales for the generalized Kuramoto system of oscillators.

\begin{lemma}[{\cite[Lemma 3]{Neunzert1978}}]\label{lem:LipschitzWasserstein}
	Let $T\colon X\to X$ be a surjective and Lipschitz continuous mapping with Lipschitz constant $L$ and let $\mu,\nu\in\mathcal P(X)$. Then,
	\begin{align*}
		W_1(T\#\mu, T\#\nu) \le \max(1,L)\ W_1(\mu,\nu).
	\end{align*}
\end{lemma}

The next result is sometimes also referred to as a Dobrushin-type estimate:

\begin{theorem}\label{thm:CtsDepIni}
	Let $\mu^\mathrm{in},\nu^\mathrm{in}\in\mathcal P(\S)^M$ be two different initial measures and let $\mu,\nu\colon [0,T] \to \mathcal P(\S)^M$ be the solutions to system \eqref{eq:SystemOfEquations}--\eqref{eq:coupling} with $\mu(0)=\mu^\mathrm{in}, \nu(0)=\nu^\mathrm{in}$. Then, for $t>0$ we have
	\begin{align*}
		\sum_{\sigma = 1}^M W_1(\mu_\sigma(t), \nu_\sigma(t)) \le e^{L\left(\sum_{\sigma=1}^{M}\bar{s^\sigma}\right)t+Lt} \sum_{\sigma = 1}^M W_1(\mu^\mathrm{in}_\sigma, \nu_\sigma^\mathrm{in}).
	\end{align*}
\end{theorem}

\begin{proof}
	Let $\pi_\sigma(t) := T_{t,0}^\sigma[\nu] \# \mu_\sigma^\mathrm{in}$ for $t\in [0,T]$. Then,
	\begin{align*}
		\sum_{\sigma = 1}^M W_1(\mu_\sigma(t), \nu_\sigma(t)) \le \sum_{\sigma = 1}^M\Big( W_1(\mu_\sigma(t), \pi_\sigma(t)) + W_1(\pi_\sigma(t), \nu_\sigma(t))\Big).
	\end{align*}
	By using Lemma \ref{lem:estimate}, Lemma \ref{lem:Wasserstein_sup} and the notation introduced in~\eqref{eq:w_def} we see, that
	\begin{align*}
		W_1(\mu_\sigma(t), \pi_\sigma(t)) &= W_1( T_{t,0}^\sigma[\mu] \#\mu^\mathrm{in}_\sigma, T_{t,0}^\sigma[\nu] \#\mu^\mathrm{in}_\sigma )\\
		&\le \sup_{\phi\in\S} \snorm{T_{t,0}^\sigma[\mu]\phi-T_{t,0}^\sigma[\nu]\phi}\\
		&\le L\bar{s^\sigma} e^{Lt}\int_0^t w(\tau) e^{-\tau L}\ \d \tau.
	\end{align*}
	By Lemma \ref{lem:LipschitzWasserstein}, we obtain
	\begin{align*}
		W_1(\pi_\sigma(t), \nu_\sigma(t)) &= W_1(   T_{t,0}^\sigma[\nu] \# \mu_\sigma^\mathrm{in} , T_{t,0}^\sigma[\nu]\#\nu_\sigma^\mathrm{in})\\
		&\le \max(1,L_T)\ W_1(\mu_\sigma^\mathrm{in}, \nu_\sigma^\mathrm{in}),
	\end{align*}
	where $L_T$ is the Lipschitz constant of $T_{t,0}^\sigma[\nu]$. By Corollary \ref{cor:LipschitzFlow}, we have $L_T = e^{Lt}$.
	Putting everything together, we get
	\begin{align*}
		\sum_{\sigma = 1}^M W_1(\mu_\sigma(t),\nu_\sigma(t)) &\le \sum_{\sigma=1}^M \Bigg[ L\bar{s^\sigma}e^{Lt} \int_0^t w(\tau) e^{-\tau L} \d \tau + e^{Lt}W_1(\mu^\mathrm{in}_\sigma, \nu^\mathrm{in}_\sigma) \Bigg]\\
		&=L\left(\sum_{\sigma=1}^{M}\bar{s^\sigma}\right) e^{Lt}\int_0^t w(\tau) e^{-\tau L } \d \tau+e^{Lt}w(0).
	\end{align*}
	Now, we can rewrite the previous inequality as follows:
	\begin{align*}
		w(t)\le e^{Lt}w(0) + L\left(\sum_{\sigma=1}^{M}\bar{s^\sigma} \right) e^{Lt}\int_0^t w(\tau) e^{-\tau L}\ \d \tau.
	\end{align*}
	Dividing by $e^{Lt}$ and then applying Gronwall's Lemma yields that
	\begin{align*}
		w(t)\le w(0) e^{L\left(\sum_{\sigma=1}^{M}\bar{s^\sigma}\right)t+Lt},
	\end{align*}
	so the result follows.
\end{proof}

\subsection{Special Initial Measures}

\subsubsection{Vlasov-Fokker-Planck Mean-Field Equation for Absolutely Continuous Measures}

As a first special case, we can state the mean-field limit PDE, which is given by the initial value problem
\begin{align}\label{eq:initialvalue}
	\frac{\partial}{\partial t}\rho_\sigma(t,\phi) + \frac{\partial }{\partial \phi}\Big( V_\sigma[\rho](t,\phi) \rho_\sigma(t,\phi)\Big) = 0,
\end{align}
with
\begin{align*}
	V_\sigma[\rho](t,\phi) = (\mathcal K_\sigma \mu_{\rho(t)})(\phi)
\end{align*}
and $\sigma \in [M]$. Here, $\mu_{\rho(t)}\in \mathcal P(\S)^M$ is the collection of measures whose densities are given by $(\rho_\sigma(t,\cdot))_{\sigma \in [M]}$. The initial conditions
\begin{align}\label{eq:initialvalueini}
	\rho_\sigma(0,\phi) = \rho^\mathrm{in}_\sigma(\phi)
\end{align}
are given such that
\begin{align*}
	\int_\S \rho^\mathrm{in}_\sigma(\phi) \ \d\phi = 1.
\end{align*}
We already remark that we are eventually going to show below the natural interpretation that $\rho_\sigma(t,\phi)$ describes the probability of finding an oscillator of population $\sigma$ at time $t$ at a position $\phi$.

\begin{definition}[Weak solution]

A collection of measurable functions $(\rho_\sigma)_{\sigma \in [M]}$ with $\rho_\sigma\colon [0,T]\times\S\to\R$ is a \emph{weak solution} of the initial value problem \eqref{eq:initialvalue}--\eqref{eq:initialvalueini} if the following two conditions are fulfilled:
\begin{itemize}
	\item For every $f\in C(\S)$ and for all $\sigma \in [M]$, the maps $t\mapsto \int_\S f(\alpha)\rho_\sigma(t, \alpha) \d \alpha$ are continuous.
	\item For all $\sigma\in[M]$ and $w_\sigma\in C^1([0,T] \times\S)$ with compact support in $[0,T)\times\S$, the following identity holds
	\begin{align*}
		 \int_0^T\int_\S \rho_\sigma(t,\phi)\left(\frac{\partial}{\partial t} w_\sigma(t,\phi) + V_\sigma[\rho](t,\phi) \frac{\partial}{\partial \phi}w_\sigma(t,\phi)\right) \d \phi \d t+\int_\S \rho^\mathrm{in}_\sigma(\phi)w_\sigma(0,\phi)\ \d \phi = 0.
	\end{align*}
\end{itemize}

\end{definition}

\begin{theorem}\label{thm:ExistenceInitialValueProblem}
	For given $\rho^\mathrm{in}_1,\dots, \rho^\mathrm{in}_M\in L^1(\S)$ with $\int \rho_\sigma^\mathrm{in} = 1$, the initial value problem \eqref{eq:initialvalue}--\eqref{eq:initialvalueini} has a weak solution.
\end{theorem}
\begin{proof}
	We consider the measures $\mu_\sigma^\mathrm{in}$ given by $\mu_\sigma^\mathrm{in}(A) = \int_A \rho_\sigma^\mathrm{in}(\phi)~\d \phi$. Theorem \ref{thm:ExistenceMeasure} yields the existence of a unique solution $\mu(t) = (\mu_1(t),\dots,\mu_M(t))$ for the system \eqref{eq:SystemOfEquations}--\eqref{eq:coupling}, with $\mu_\sigma(t) =T^\sigma_{t,0}[\mu] \# \mu^\mathrm{in}_\sigma$. Since Corollary~\ref{cor:LipschitzFlow} gives us Lipschitz continuity of $T^\sigma_{t,0}[\mu]$, Rademacher's Theorem tells us that it is almost everywhere differentiable with essentially bounded weak derivative \cite[Chapter 5.8, Theorem 4]{Evans2010}. Now, for any measurable set $A\subset \S$, the change-of-variables formula yields
	\begin{align*}
		\mu_\sigma(t)(A) = \mu_\sigma^\mathrm{in} \circ T^\sigma_{0,t}[\mu]A = \int_{T^\sigma_{0,t}[\mu]A} \rho^\mathrm{in}_\sigma(\phi)\ \d \phi = \int_A \rho^\mathrm{in}_\sigma(T^\sigma_{0,t}[\mu]\phi)\left\lvert\frac{\partial}{\partial \phi}T^\sigma_{0,t}[\mu]\phi\right\rvert\ \d\phi.
	\end{align*}
	So $\mu_\sigma(t)$ is absolutely continuous with density
	\begin{align*}
		\rho_\sigma(t,\phi) = \rho^\mathrm{in}_\sigma(T^\sigma_{0,t}[\mu]\phi)\left\lvert\frac{\partial}{\partial \phi}T^\sigma_{0,t}[\mu]\phi\right\rvert.
	\end{align*}
	To show that $(\rho_\sigma)_{\sigma \in [M]}$ is a weak solution of the initial value problem, we take $w_\sigma \in C^1([0,T]\times \S)$ with compact support in $[0,T)\times\S$ and define
	\begin{align*}
		h_\sigma(t,\phi) := \frac{\partial}{\partial t}w_\sigma(t,\phi) + V_\sigma[\rho](t,\phi)\frac{\partial}{\partial \phi}w_\sigma(t,\phi),\quad \sigma \in [M].
	\end{align*}
	Now, we calculate
	\begin{align*}
		\int_0^T\int_\S \rho_\sigma(t,\phi)h_\sigma(t,\phi)\ \d\phi\d t &= \int_0^T\int_\S\rho^\mathrm{in}_\sigma(T^\sigma_{0,t}[\mu]\phi)\left\lvert\frac{\partial}{\partial \phi}T^\sigma_{0,t}[\mu]\phi\right\rvert h_\sigma(t,\phi) \ \d \phi\d t\\
		&= \int_\S \rho_\sigma^\mathrm{in}(\phi)\int_0^T h_\sigma(t,T_{t,0}^\sigma[\mu]\phi)\ \d t\d\phi.
	\end{align*}
	By using the chain rule, it turns out, that $h_\sigma(t,T_{t,0}^\sigma[\mu]\phi)$ is the derivative of $w_\sigma(t,T_{t,0}^\sigma[\mu]\phi)$:
	\begin{align*}
		\frac{\d}{\d t} w_\sigma(t,T^\sigma_{t,0}[\mu]\phi) &= \frac\partial{\partial t} w_\sigma(t,T^\sigma_{t,0}[\mu]\phi) + \frac{\partial}{\partial \phi} w_\sigma(t,T^\sigma_{t,0}[\mu]\phi) \frac{\partial}{\partial t}T^\sigma_{t,0}[\mu]\phi\\
		&=\frac{\partial}{\partial t} w_\sigma(t,T^\sigma_{t,0}[\mu]\phi) + \frac{\partial}{\partial\phi} w_\sigma(t,T^\sigma_{t,0}[\mu]\phi)(\mathcal K_\sigma\mu(t))(\phi)\\
		&=\frac{\partial}{\partial t} w_\sigma(t,T^\sigma_{t,0}[\mu]\phi) + \frac{\partial}{\partial \phi} w_\sigma(t,T^\sigma_{t,0}[\mu]\phi)V_\sigma[\rho](t,\phi) = h_\sigma(t,T_{t,0}^\sigma[\mu]\phi).
	\end{align*}
	Putting these two calculations together, we obtain
	\begin{align*}
		\int_0^T\int_\S \rho_\sigma(t,\phi)h_\sigma(t,\phi)\ \d\phi\d t &= \int_\S \rho_\sigma^\mathrm{in}(\phi)\int_0^T \frac{\d}{\d t} w_\sigma(t,T^\sigma_{t,0}[\mu]\phi) \ \d t\d\phi\\
		&= -\int_\S\rho^\mathrm{in}_\sigma(\phi) w_\sigma(0,\phi)\ \d\phi.
	\end{align*}
	This shows that the collection $(\rho_\sigma)_{\sigma\in[M]}$ is a weak solution of the initial value problem \eqref{eq:initialvalue}--\eqref{eq:initialvalueini}.
\end{proof}

\begin{remark}
	By following and adapting the calculations in \cite[Theorem 2.3.6 and Theorem 3.3.4]{Golse2013} one can even show the uniqueness of solutions to the initial value problem \eqref{eq:initialvalue}--\eqref{eq:initialvalueini}.
\end{remark}

\begin{remark}
	In case the initial conditions $\rho_\sigma^\mathrm{in}(\phi)$ and the coupling functions $G_\sigma$ are sufficiently smooth, it follows from standard ODE theorems, that the flow $T_{t,0}^\sigma[\mu](\phi)$ and thus also the densities $\rho_\sigma(t,\phi)$ at later times are smooth with respect to $\phi$.
\end{remark}

\subsubsection{Discrete Initial Measures}

As a further special case of the characteristic system \eqref{eq:SystemOfEquations}--\eqref{eq:coupling}, one is often interested in choosing $\mu^\textrm{in}$ as a discrete measure. In particular, initial measures of the form
\begin{align}\label{eq:discrete_ini_measure}
	\mu^\textrm{in}_\sigma = \frac{1}{N}\sum_{k=1}^{N}\delta_{\phi^\textrm{in}_{\sigma,k}},\qquad \sigma\in [M],
\end{align}
describe the discrete states of $N$ oscillators in each of the $M$ populations.
It can easily be seen, that if functions $\phi_{\sigma,k}\colon \R_{\ge 0} \to \S$ solve the generalized Kuramoto ODE system
\begin{align}\label{eq:discrete_system}
	\dot\phi_{\sigma,k} = \omega_\sigma + \frac{1}{N^{\abs{s^\sigma}}} \sum_{i\in [N]^{\abs{s^\sigma}}}  G_\sigma(\phi_{(s^\sigma_1),i_1},\dots,\phi_{(s^\sigma_{\abs{s^\sigma}}), i_{\abs{s^\sigma}}}, \phi_{\sigma,k})
	\end{align}
with initial condition $\phi_{\sigma,k}(0) = \phi^\mathrm{in}_{\sigma,k}$, then the measures
\begin{align}\label{eq:discrete_measure}
	\mu_\sigma(t) = \frac{1}{N}\sum_{k=1}^{N}\delta_{\phi_{\sigma,k}(t)}
\end{align}
solve the characteristic system~\eqref{eq:SystemOfEquations}--\eqref{eq:coupling}.

\subsubsection{Relation between Absolutely Continuous and Discrete Initial Measures}

In order to find a relation between absolutely continuous initial measures and discrete initial measures in the characteristic system \eqref{eq:SystemOfEquations}--\eqref{eq:coupling} we apply Theorem \ref{thm:CtsDepIni} for the case of $\mu^\mathrm{in}\in \mathcal P(\S)^M$ consisting of discrete measures and the components of~$\nu^\mathrm{in}\in \mathcal P(\S)^M$ being absolutely continuous measures. In fact we even consider a whole family of discrete initial measures given by
\begin{align}\label{eq:initial_cond_N}
	\mu_\sigma^{\textrm{in}, N} = \frac{1}{N}\sum_{k= 1}^{\Omega_{N,\sigma}}\delta_{\phi^{\textrm{in},N}_{\sigma,k}},\qquad \sigma\in [M],
\end{align}
where $N$ parameterizes the family $(\mu_\sigma^{\textrm{in},N})_{N \in\N}$ and $\Omega_{N,\sigma}$ denotes the amount of discrete oscillators in population $\sigma$ for $\mu_\sigma^{\textrm{in},N}$. Suppose that this family of measures converges for each~$\sigma\in[M]$ as $N\rightarrow \infty$ in the Wasserstein-1 distance to measures~$\nu_\sigma^\mathrm{in}$, which can be represented by densities~$\rho^\mathrm{in}_\sigma$. The solution of the original system \eqref{eq:SystemOfEquations} for initial conditions given by \eqref{eq:initial_cond_N} can be obtained by solving the ODE \eqref{eq:discrete_ini_measure}--\eqref{eq:discrete_system}, whereas the solution of this system for initial condition given by the densities $\rho^\mathrm{in}_\sigma$ can be obtained by solving the mean-field limit \eqref{eq:initialvalue}--\eqref{eq:initialvalueini}. The mean-field limit is analytically much easier than systems of the form~\eqref{eq:discrete_system} with large~$N$. Therefore, one typically tries to find results for the mean-field limit \eqref{eq:initialvalue} first. Then, one employs Theorem \ref{thm:CtsDepIni} with $\mu^{\textrm{in},N}_\sigma$ for $\mu^\mathrm{in}$ and $\nu^\mathrm{in}$ being the limit of $\mu^{\textrm{in},N}_\sigma$ as $N\to \infty$. This theorem gives an upper bound of how well the mean-field limit approximates solutions to the finite $N$ system \eqref{eq:discrete_system}. In fact, Theorem \ref{thm:CtsDepIni} tells us that whenever a sequence of discrete initial measures \eqref{eq:initial_cond_N} converges to absolutely continuous initial measures, represented by densities $\rho^\mathrm{in}_\sigma$, then the solutions of system \eqref{eq:discrete_system}--\eqref{eq:discrete_measure} will also converge at a fixed later time $t$ to densities $\rho_\sigma(t,\cdot)$ which solve the initial value problem \eqref{eq:initialvalue}--\eqref{eq:initialvalueini}.

\subsection{Specific coupling functions}
\label{sec:Sinusoid}

Now that we have set up a general framework for multi-population oscillators with higher-order coupling, it is important to stress that the additional terms beyond the classical Kuramoto model, can also induce new dynamical effects. This is one main motivation to keep the theory as general as possible. Yet, to see that new effects occur, it is interesting to consider specific coupling functions, so that the dynamics of the characteristic equation~\eqref{eq:SystemOfEquations} are more restricted.

To illustrate the emergence of new dynamics, we first consider finite networks of~$N$ phase oscillators with sinusoidal coupling. As shown in~\cite{Watanabe1994}, a trajectory of a coupled oscillator system of the form
\begin{align}\label{eq:WatanabeStrogatzSystem}
    \dot \phi_k = f(\phi) + g(\phi) \cos(\phi_k) + h(\phi) \sin(\phi_k), \qquad k = 1,\dots, N
\end{align}
where $\phi_k\in \S$ is the phase of each oscillator and $f,g,h\colon \S^N\to \R$ are functions of $\phi = (\phi_1,\dots,\phi_N)$, does not explore the whole $N$-dimensional phase space, but is rather constrained to a three-dimensional manifold. In particular, when employing the transformation
\begin{align}\label{eq:WatanabeStrogatzTransformation}
    \tan\left(\frac 12 (\phi_k(t) - \Theta(t))\right) = \sqrt{\frac{1+\gamma(t)}{1-\gamma(t)}}\tan\left(\frac 12 (\psi_k - \Psi(t))\right), \quad k = 1,\dots, N,
\end{align}
where $\gamma(t), \Theta(t)$ and $\Psi(t)$ are functions satisfying the $3$-dimensional ODE system
\begin{subequations}
\label{eq:WatanabeStrogatzReducedODE}
\begin{align}
    \dot\gamma &= -(1-\gamma^2)(g(\phi)\sin(\Theta) - h(\phi) \cos(\Theta)),\\
    \gamma \dot \Psi &= -\sqrt{1-\gamma^2}(g(\phi)\cos(\Theta) + h(\phi)\sin(\Theta)),\\
    \gamma \dot \Theta &= \gamma f(\phi) - g(\phi) \cos(\Theta) - h(\phi) \sin(\Theta),
\end{align}
\end{subequations}
the coefficients $\psi_1,\dots,\psi_N$ in~\eqref{eq:WatanabeStrogatzTransformation} are actually constant in time~\cite{Watanabe1994}. This degeneracy is a consequence of the sinusoidal coupling and can be understood in terms of M\"obius group actions~\cite{Marvel2009,Stewart2011}.
Even though~\eqref{eq:WatanabeStrogatzSystem} seems quite restrictive, it is actually contains a large class of systems going beyond the Kuramoto~\cite{Chen2017}; see also~\cite{Bick2018c}.
For example, as one can see by applying trigonometric identities, the system~\eqref{eq:System_Skardal} from the introduction belongs to this class. When converting an initial condition $\phi(0) = (\phi_1(0),\dots,\phi_N(0))$ of~\eqref{eq:WatanabeStrogatzSystem} into initial conditions $\gamma(0), \Psi(0), \Theta(0)$ of~\eqref{eq:WatanabeStrogatzReducedODE} and constants $\psi_1,\dots,\psi_N$, there are~$N$ equations~\eqref{eq:WatanabeStrogatzTransformation}, but~$N+3$ variables to choose. Therefore, three degrees of freedom are left, and one typically chooses the constants $\psi_1,\dotsc,\psi_N$ such that
\begin{align}\label{eq:WatanabeStrogatz_psiRestriction}
    \frac 1N \sum_{k=1}^N \cos(\psi_k) = \frac 1N \sum_{k=1}^N \sin(\psi_k) = \frac{1}{N} \sum_{k=1}^N \psi_k = 0.
\end{align}
It was shown in~\cite{Watanabe1994}, that one can find suitable parameters $\psi_1,\dotsc,\psi_N$ satisfying~\eqref{eq:WatanabeStrogatz_psiRestriction} and~$\gamma(0), \Psi(0)$ and $\Theta(0)$ for almost all initial conditions $\phi(0) = (\phi_1,\dots,\phi_N)$. In particular, the only exception is when there are so called ``majority clusters", when the phases of at least~$N/2$ oscillators coincide.

By imposing the restrictions~\eqref{eq:WatanabeStrogatz_psiRestriction}, one can think of the parameters $\psi_1,\dots,\psi_N$ as the ones foliating the $N$-dimensional phase space into a~$N-3$ dimensional continuum of $3$-dimensional invariant manifolds. The dynamics inside these manifolds is then described by $\gamma(t), \Psi(t)$ and~$\Theta(t)$.

Let us now apply this theory to the system~\eqref{eq:System_Skardal}, that we considered in the introduction, with $K_2 = 0$:
\begin{align}\label{eq:skardal_arenas_motivation}
    \dot \phi_k = \omega + \frac{K_1}{N}\sum_{j=1}^N \sin(\phi_j-\phi_k) + \frac{K_3}{N^3} \sum_{j=1}^N \sum_{l = 1}^N\sum_{m=1}^N \sin(\phi_j + \phi_l - \phi_m - \phi_k)
\end{align}
Here, $k = 1,\dots, N$, $\omega\in \R$ is the intrinsic frequency of all oscillators and $K_1, K_3\in \R$ are coupling constants. By introducing the order parameter 
\begin{align*}
    re^{i\alpha} = \frac{1}{N}\sum_{j=1}^N e^{i\phi_j},
\end{align*}
this system can also be written as
\begin{align}\label{eq:SkardalArenasSystemReformulation}
    \dot \phi_k = \omega + \frac{K_1+r^2 K_3}{2i}\left( r e^{-i\phi_k} - \bar r e^{i\phi_k}\right).
\end{align}
The equations \eqref{eq:WatanabeStrogatzReducedODE}, determining the evolution of $\gamma(t), \Psi(t)$ and $\Theta(t)$, then turn into
\begin{subequations}
\label{eq:WatanabeStrogatzReducedDynamics}
\begin{align}
    \dot \gamma &= (1-\gamma^2)(K_1 + K_3 r^2) \frac 1N \sum_{j=1}^N \frac{\gamma-\cos(\psi_j-\Psi)}{1-\gamma\cos(\psi_j-\Psi)},\\
    \gamma \dot \Psi &= - (1-\gamma^2) (K_1 + K_3 r^2)\frac 1N \sum_{j=1}^N \frac{\sin(\psi_j-\Psi)}{1-\gamma \cos(\psi_j-\Psi)},\\
    \gamma \dot\Theta &= \gamma \omega - (K_1 + K_3 r^2) \frac 1N \sum_{j=1}^N \frac{\sqrt{1-\gamma^2}\sin(\psi_j-\Psi)}{1-\gamma\cos(\psi_j-\Psi)}.
\end{align}
\end{subequations}
Since the order parameter fulfills
\begin{align*}
    r^2 = \frac 1{N^2} \sum_{j,k=1}^N \frac{(\cos(\psi_j-\Psi)-\gamma)(\cos(\psi_k-\Psi)-\gamma) + (1-\gamma^2) \sin(\psi_j-\Psi)\sin(\psi_k-\Psi)}{(1-\gamma\cos(\psi_j-\Psi))(1-\gamma\cos(\psi-\Psi))},
\end{align*}
as one can see by using trigonometric identities (see~\cite[Eq.~(3.3)]{Watanabe1994}), it is a function of~$\gamma$ and~$\Psi$ only. Therefore,~$\Theta$ does not appear on the right-hand side of~\eqref{eq:WatanabeStrogatzReducedDynamics} and consequently, the dynamics of~\eqref{eq:WatanabeStrogatzReducedDynamics} is essentially two dimensional. Analogously to~\cite{Watanabe1994}, we can now define a potential function
\begin{align*}
    \mathcal H (\Psi, \gamma) := \frac 1N \sum_{j=1}^N \log\left(\frac{1-\gamma \cos(\psi_k-\Psi)}{\sqrt{1-\gamma^2}}\right).
\end{align*}
As calculated in~\cite{Watanabe1994}, the partial derivatives of $\mathcal H$ are given by
\begin{align*}
    \frac{\partial \mathcal H}{\partial \Psi} = -\frac 1N \sum_{k=1}^N \frac{\gamma \sin(\psi_k - \Psi)}{1-\gamma \cos(\psi_k-\Psi)}, \qquad
    \frac{\partial \mathcal H}{\partial \gamma} = \frac{1}{N(1-\gamma^2)}\sum_{k=1}^N \frac{\gamma-\cos(\psi_k-\Psi)}{1-\gamma\cos(\psi_k-\Psi)}.
\end{align*}
This allows us to rewrite the $(\gamma, \Psi)$-dynamics of~\eqref{eq:WatanabeStrogatzReducedDynamics} in terms of partial derivatives of $\mathcal H$:
\begin{align*}
    \dot\gamma &= (1-\gamma^2)^2\frac{\partial \mathcal H}{\partial \gamma} (K_1 + K_3 r^2)\\
    \dot \Psi &= \frac{1-\gamma^2}{\gamma^2}\frac{\partial \mathcal H}{\partial \Psi} (K_1 + K_3 r^2).
\end{align*}
Using this, we can now calculate
\begin{align*}
    \dot{\mathcal H} &= \frac{\partial \mathcal H}{\partial \gamma} \dot\gamma + \frac{\partial \mathcal H}{\partial \Psi} \dot \Psi\\
    &=(1-\gamma^2)\left[ (1-\gamma^2)\left( \frac{\partial \mathcal H}{\partial \gamma}\right) ^2 + \frac{1}{\gamma^2}\left( \frac{\partial \mathcal H}{\partial \Psi}\right)^2 \right] (K_1 + K_3 r^2).
\end{align*}
However, since
\begin{align*}
    (1-\gamma^2)\left[ (1-\gamma^2)\left( \frac{\partial \mathcal H}{\partial \gamma}\right) ^2 + \frac{1}{\gamma^2}\left( \frac{\partial \mathcal H}{\partial \Psi}\right)^2 \right] = r^2,
\end{align*}
as calculated in~\cite{Watanabe1994}, we obtain $\dot{\mathcal H} = r^2(K_1 + K_3 r^2)$.
At this point it is important to understand the meaning of~$\mathcal H$. As explained in~\cite{Watanabe1994}, on the one hand
\begin{align}\label{eq:H_interpret1}
    r \to 1 \Leftrightarrow \gamma \to 1 \Leftrightarrow \mathcal H \to \infty
\end{align}
and on the other hand
\begin{align}\label{eq:H_interpret2}
    r = 0 \Leftrightarrow \gamma = 0 \Leftrightarrow \mathcal H = 0.
\end{align}
Consequently, whenever $K_1 + K_3 r^2>0$ and $r>0$, $\mathcal H$ increases until either $r = 1$ or $K_1 + K_3 r^2 = 0$ in the limit. Conversely, if $K_1 + K_3 r^2<0$ and $r>0$, $\mathcal H$ decreases until $r=0$ or $K_1 + K_3 r^2 = 0$.

\begin{example}\label{exam:WatanabeStrogatzExample}
Let $K_1 = 1, K_3 = -4$ and the initial condition~$\phi$ be given such that there are no majority clusters and $0<r<1/2$. $\mathcal H$ is stationary only if $r=0$ or $r = \sqrt{-K_1/K_3} = 1/2$. Otherwise~$\mathcal H$ is increasing. Thus, there are only two possibilities: Firstly, $\mathcal H$ can diverge to infinity or secondly $\lim_{t\to \infty} \mathcal H$ exists in~$\R$. In the first case, $r\to 1$, by~\eqref{eq:H_interpret1} and consequently~$r$ has to achieve the value~$1/2$ at an intermediary time. This, however, cannot happen, since at $r=1/2$, all oscillators of~\eqref{eq:SkardalArenasSystemReformulation} are only rotating around the circle with a common frequency. So $r=1/2$ is invariant and thus $r\to 1$ is impossible. Consequently $\lim_{t\to \infty}\mathcal H$ exists in~$\R$, which in turn can only happen if $r\to 1/2$. To conclude this example, almost all initial conditions display $r=1/2$ in the limit as $t\to\infty$, and the oscillators consequently partially synchronize. This occurs even though the intrinsic frequency~$\omega$ in~\eqref{eq:skardal_arenas_motivation} does not depend on~$k$ as it is typically the case for the classical Kuramoto model when one speaks of partial synchronization. Moreover, this observation is in line with the emergence of new stationary phase configurations when considering ``nonlinear coupling'' through the order parameter~\cite{Chen2017} beyond the Kuramoto model.
\end{example}

While the above analysis applies only to discrete initial measures, a similar analysis can be conducted when the initial measure has a density $\rho^\text{in}(\phi)$. In this case, the evolution of~$\rho$ is governed by the continuity equation. Then, there are no constants $\psi_1,\dots,\psi_N$, but there is a density $\chi(\psi)$ on the circle that does not depend on time. The transformation to get from~$\phi$ to $\psi$ is exactly the same as described in~\eqref{eq:WatanabeStrogatzTransformation} and the requirement for the density~$\chi(\psi)$ to be independent of time is given by~\eqref{eq:WatanabeStrogatzReducedODE}, as well. Since the initial density $\rho$ cannot have any majority clusters, it is always possible to find initial conditions for $\gamma, \Psi, \Theta$ such that the density $\chi$ satisfies
\begin{align*}
    \int_\S \cos(\psi)\chi(\psi)\ \d\psi = \int_\S \sin(\psi)\chi(\psi)\ \d\psi  = \int_\S \psi\ \chi(\psi)\ \d\psi = 0.
\end{align*}
In general, sums that depend on $(\psi_k)_{k=1,\dots,N}$ and appear when dealing with discrete measures have to be replaced by integrals with respect to the density $\chi(\phi)$ in the absolutely continuous case. Therefore, the potential function $\mathcal H$ reads as
\begin{align*}
    \mathcal H (\Psi, \gamma) = \int_\S \log\left( \frac{1-\gamma \cos(\phi-\Psi)}{\sqrt{1-\gamma^2}}\right) \chi(\psi) \ \d\psi
\end{align*}
and $\dot{\mathcal H} = r^2(K_1 + K_3 r^2)$ still holds true. Consequently, the explanation in Example~\ref{exam:WatanabeStrogatzExample} can be adopted to absolutely continuous measures.

As we have seen in the previous examples, the order parameter can give insights about the degree of synchrony in an oscillator system. The dynamics of this order parameter has interesting bifurcations at $r=0$ and $r=1$ but there can also be further equilibria of $r\in (0,1)$. For instance, in Example~\ref{exam:WatanabeStrogatzExample}, $r=0$ and $r=1$ were repelling equilibria whereas $r=1/2$ was attracting. It is therefore interesting under which conditions certain invariant states of an oscillator system are stable or unstable. Since $r=1$ in each population is the only invariant state of the general characteristic oscillator system~\eqref{eq:SystemOfEquations} and under minor assumptions, $r=0$ includes a general invariant state, as well, the next section is devoted to a stability analysis of these two states.

\section{Synchrony and Synchrony Patterns and their Stability}\label{sec:dynamics}

Of course, instead of working with the simpler mean-field limit, or with a particular finite-dimensional Kuramoto model, an alternative route is to directly study the characteristic system. This system links finite-dimensional oscillator systems to the mean-field. In particular, it contains the full information about both systems, so we start our dynamical analysis here using the characteristic system.

From now on we assume the multi-indices~$s^\sigma$ to have the form
\begin{align}\label{eq:multiindex_specialform}
	s^\sigma = (a^\sigma_1,a^\sigma_1, a^\sigma_2, a^\sigma_2, \dots, a^\sigma_{L_\sigma}, a^\sigma_{L_\sigma}, \sigma)
\end{align}
for some $L_\sigma\in \N$ and $a^\sigma_j\in [M]$, i.e., $\abs{s^\sigma} = 2L_\sigma +1$. Further, the coupling functions $G_\sigma\colon \S^{\abs{s^\sigma}}\times\S\to \R$ are supposed to be of the form $G_\sigma(\alpha,\phi) = g_\sigma(\alpha_1-\alpha_2,\dots, \alpha_{2L_\sigma-1}-\alpha_{2L_\sigma}, \alpha_{2L_\sigma+1}-\phi)$, for functions $g_\sigma\colon \S^{L_\sigma}\times\S\to \R$. These conditions can be summarized by requiring the velocity field $\mathcal K_\sigma\mu$ to be given by
\begin{align}\label{eq:coupling_differences}
	(\mathcal K_\sigma\mu)(\phi) = \omega_\sigma + \int_\S \int_{\S^{\abs{r^\sigma}}} \int_{\S^{\abs{r^\sigma}}} g_\sigma(\alpha-\beta, \gamma-\phi) \ \d\mu^{(r^\sigma)}(\alpha)\d\mu^{(r^\sigma)}(\beta)\d\mu_\sigma(\gamma),
\end{align}
for new multi-indices $r^\sigma = (a^\sigma_1,\dots, a^\sigma_{L^\sigma})$ with $\abs{r^\sigma} = L_\sigma$ and coupling functions $g_\sigma\colon \S^{\abs{r^\sigma}}\times\S\to\R$. In the remainder of this paper we study the system \eqref{eq:SystemOfEquations} with \eqref{eq:coupling_differences}.

\newcommand{\Sp}{\mathrm{S}}
\newcommand{\Dp}{\mathrm{D}}

\newcommand{\rset}[2]{\left\lbrace\, #1\,\left|\;#2\right.\right\rbrace}
\newcommand{\lset}[2]{\left\lbrace\left. #1\;\right|\,#2\,\right\rbrace}
\newcommand{\set}[2]{\rset{#1}{#2}}
\newcommand{\tset}[2]{\big\lbrace #1\,\big|\;#2\big\rbrace}
\newcommand{\sset}[1]{\left\lbrace #1\right\rbrace}
\newcommand{\tsset}[1]{\big\lbrace #1\big\rbrace}

\begin{definition}
We write $\Dp = \sset{\frac{1}{2\pi}\lambda_\S}$, where~$\lambda_\S$ is the Hausdorff measure on~$\S$, and $\Sp = \set{\delta_\xi}{\xi\in\S}$. Population~$\sigma\in[M]$ is in \emph{splay phase}~$\Dp$ if $\mu_\sigma \in\Dp$ and \emph{phase synchronized} if $\mu_\sigma \in\Sp$.
\end{definition}

We adopt the notation in~\cite{Bick2019a} and write~$\Dp$ if a population is in splay configuration and~$\Sp$ if it is phase synchronized.
That is, we write
\begin{subequations}\label{eq:SyncSplay}
\begin{align}
\mu_1\dotsb\mu_{\sigma-1}\Sp\mu_{\sigma+1}\dotsb\mu_{M} &= \lset{\mu\in\mathcal{P}(\S)^N}{\mu_\sigma\in\Sp},\\
\mu_1\dotsb\mu_{\sigma-1}\Dp\mu_{\sigma+1}\dotsb\mu_{M} &= \lset{\mu\in\mathcal{P}(\S)^N}{\mu_\sigma\in\Dp}
\end{align}
\end{subequations}
to indicate that population~$\sigma$ is fully phase synchronized or in splay phase. We extend the notation to intersections of the sets~\eqref{eq:SyncSplay}. Consequently, $\Sp\dotsb\Sp$ ($M$~times) is the set of cluster states where all populations are fully phase synchronized and $\Dp\dotsb\Dp$ the set where all populations are in splay phase.

\subsection{Invariant Subspaces}

\begin{proposition}[Reducibility to lower dimensions]\label{prop:reducibility}
	If we fix~$m$ populations, each to be in splay or in synchronized state, the other $M-m$ populations behave accordingly to \eqref{eq:SystemOfEquations},\eqref{eq:coupling_differences} with $M-m$ instead of $M$ and different coupling functions.
\end{proposition}

\begin{proof}
	Without loss of generality we can assume that $m=1$, since the case for general~$m$ then follows by repeatedly applying this proposition. After reindexing populations, we can also assume the $M$th population to be fixed in synchronized or splay state. For a convenience of notation we suppose the multi-indices $r^\sigma$ to be sorted in ascending order, which can easily be achieved by changing the order of integration. Let us now denote $\chi^\sigma =  \abs{\{ i:r^\sigma_i=M\}}$ and write $v^\sigma\in [M]^{p^\sigma}$ with $p^\sigma = \abs{r^\sigma} - \chi^\sigma$ for the multi-index having the same entries as $r^\sigma$ except that the last $\chi^\sigma$ entries, i.e., all entries valued $M$, are missing.
	In case we fix the $M$th population to be synchronized, the other $M-1$ populations yield a system of the form~\eqref{eq:SystemOfEquations},\eqref{eq:coupling_differences} with coupling functions
	\begin{align*}
		\hat g_\sigma(\alpha,\phi) := g_\sigma(\alpha,0^{\chi^\sigma},\phi).
	\end{align*}
	Similarly, if the $M$th population is in splay state, the other $M-1$ populations move according to the coupling functions
	\begin{align*}
		\hat g_\sigma(\alpha,\phi) = \frac{1}{(2\pi)^{\chi^\sigma}}\int_{\S^{\chi^\sigma}}g_\sigma(\alpha,\beta, \phi) \d\beta,
	\end{align*}
	with $\alpha \in \S^{\abs{v^\sigma}}$.	Here, $g_\sigma\colon \S^{\abs{r^\sigma}}\times \S\to\R$ are considered as functions mapping from $\S^{p^\sigma}\times \S^{\chi^\sigma}\times\S$ to $\R$.
\end{proof}

\begin{proposition}\label{prop:invariant_subspaces}
	Subsets of the form \eqref{eq:SyncSplay} are invariant under the flow of \eqref{eq:SystemOfEquations},\eqref{eq:coupling_differences}.
\end{proposition}

\begin{proof}
	Without loss of generality, we only consider $\Sp\mu_2\cdots\mu_M$ and $\Dp\mu_2\cdots\mu_M$. Suppose, that $\mu^\mathrm{in}_1 = \delta_\xi$ for some $\xi\in\S$. Then, by \eqref{eq:SystemOfEquations_PushForward}, $\mu_1(t) = \Phi_1(t,\cdot,\mu^\mathrm{in})\#\delta_\xi = \delta_{\Phi_1(t,\xi,\mu^\mathrm{in})}$, so the first population always stays in a synchronized state. That proves invariance of $\Sp\mu_2\cdots\mu_M$. Proving invariance of $\Dp\mu_2\cdots\mu_M$ is a bit more involved. We start with setting $\mu_1(t) = \frac{1}{2\pi}\lambda_\S$ in \eqref{eq:SystemOfEquations} by applying Proposition \ref{prop:reducibility} and thus reducing the system by one population. Theorem \ref{thm:ExistenceMeasure} gives the existence and uniqueness of measures $\mu_2(t),\dots,\mu_M(t)$, which solve the reduced system. To see, that $\mu_1(t),\dots,\mu_M(t)$ is a solution of the unreduced system, we calculate
	\begin{align*}
		(\mathcal K_1\mu(t))(\phi)&=\omega_1 + \int_{\S^{\abs{r^\sigma}}}\int_{\S^{\abs{r^\sigma}}} \int_\S g_\sigma(\alpha-\beta,\gamma-\phi)\ \d\mu_1(\gamma)\d\mu^{(r^\sigma)}(\alpha)\d\mu^{(r^\sigma)}(\beta)\\
		&=\omega_1 + \int_{\S^{\abs{r^\sigma}}}\int_{\S^{\abs{r^\sigma}}}\frac{1}{2\pi}\int_\S g_\sigma(\alpha-\beta,\gamma-\phi)\ \d\gamma \d\mu^{(r^\sigma)}(\alpha)\d\mu^{(r^\sigma)}(\beta)\\
		&=\omega_1 + \int_{\S^{\abs{r^\sigma}}}\int_{\S^{\abs{r^\sigma}}}\frac{1}{2\pi}\int_\S g_\sigma(\alpha-\beta,\gamma)\ \d\gamma \d\mu^{(r^\sigma)}(\alpha)\d\mu^{(r^\sigma)}(\beta),
	\end{align*}
	where the last equality was based on a phase shift $\gamma-\phi\mapsto\gamma$. Therefore, the velocity field $(\mathcal K_1\mu(t))(\phi)$ is actually independent of $\phi$. As~$\Dp$ is invariant under rotations, i.e., those operations generated by constant velocity fields, $\mu_1(t),\dots,\mu_M(t)$ is a solution of the full system, which shows invariance of $\Dp\mu_2\cdots\mu_M$.
\end{proof}

\begin{remark}
	Because intersections of invariant sets are again invariant, any combination of $\Sp,\Dp$ and $(\mu_\sigma)_{\sigma\in[M]}$ is also invariant.
\end{remark}

\subsection{Stability of the All-Synchronized State}

As we have seen, the state in which one population is synchronized, is invariant. However, as the synchronized state is not really a single point, but rather a whole set of phase configurations
\begin{align*}
	\Sp = \{ \delta_\xi: \xi\in \S\},
\end{align*}

it makes more sense to study the stability of this set. Similarly, for the multi-population model, it makes sense to analyze the stability of the set
\begin{align*}
	\Sp^M := \Sp\cdots\Sp = \{ \mu \in \mathcal P(\S)^M: \mu=(\delta_{\xi_1},\dots,\delta_{\xi_M}), \xi_1,\dots,\xi_M\in \S\}.
\end{align*}

\begin{notation}
For an ease of notation, we will write 
\begin{itemize}
	\item $W_1(\Sp, \mu) = \inf_{\delta\in \Sp} W_1(\delta, \mu)$, for $\mu\in \mathcal P(\S)$,
	\item $\mu(t)\to\Sp$ as $t\to \infty$ if $\lim_{t\to\infty} W_1(\Sp,\mu(t))=0$, for $\mu \in C_{\mathcal P(\S)} $,
	\item $\mathbb B(\Sp, \epsilon) = \bigcup_{\delta\in\Sp} \mathbb B(\delta, \epsilon)$,
\end{itemize}
where $\B(\delta, \epsilon)$ denotes the ball centered at a measure $\delta$ of radius $\epsilon$ in the Wasserstein-1 metric.
\end{notation}

The following definition is the natural variant of (Lyapunov) stability in our setting:

\begin{definition}\label{def:MPopStability}
	The set $\Sp^M$ is \emph{stable} if for all $\sigma \in [M]$ and all neighborhoods $U_\sigma \subset \mathcal P(\S)$ of $\Sp$ there exist neighborhoods $V_\sigma$ of $\Sp$ such that for any $\mu^\mathrm{in} = (\mu_1^\mathrm{in},\dots,\mu_M^\mathrm{in})\in V_1\times\dots\times V_M$, the solution $\mu(t)$ of \eqref{eq:SystemOfEquations},\eqref{eq:coupling_differences} satisfies $\mu(t)\in U_1\times\dots\times U_M$ for all $t\ge 0$.
\end{definition}

Of course, one often does not only want to show stability of solutions staying near an invariant set but also that the solutions tend towards the invariant set:

\begin{definition}\label{def:MPopAsymStability}
	The set $\Sp^M$ is \emph{asymptotically stable} if it is stable and, additionally, there exists a neighborhood $V = V_1\times\dots\times V_M\subset \mathcal P(\S)^M$ such that for all $\mu^\mathrm{in}\in V$ the solution of \eqref{eq:SystemOfEquations},\eqref{eq:coupling_differences} satisfies $\mu_\sigma(t) \to \Sp$ as $t\to \infty$ for all $\sigma \in[M]$.
\end{definition}

\begin{remark}
	The two Definitions \ref{def:MPopStability} and \ref{def:MPopAsymStability} are formulated in terms of the topology created by the distance $\max_{\sigma \in[M]} W_1(\Sp, \mu_\sigma)$. However, instead of taking the maximum, one can also sum over $W_1(\Sp, \mu_\sigma)$, to get a definition that resembles the metric used to prove existence and uniqueness of \eqref{eq:SystemOfEquations} more closely. However, as these topologies are equivalent, we use Definitions \ref{def:MPopStability}-\ref{def:MPopAsymStability} in the topology generated by $\max_{\sigma \in[M]} W_1(\Sp, \mu_\sigma)$, as this setting seems easier to work with.
\end{remark}

\begin{remark}
    The two Definitions \ref{def:MPopStability} and \ref{def:MPopAsymStability} also make sense when considered with the more general coupling \eqref{eq:coupling}. However, we only work with them in the context of the velocity fields \eqref{eq:coupling_differences}.
\end{remark}

\subsubsection{No Generic Asymptotic Stability}\label{sec:NoGenericAsymptoticStability}

This section aims to illustrate that the all-synchronized state can not be asymptotically stable under a generic condition, which is in this case $\frac{\partial}{\partial\gamma}g_\sigma(0,\gamma)\neq 0$ for at least one $\sigma\in[M]$. Assume that we do not have generic asymptotic stability for $M=1$. Generalizing this to $M$ populations can then be easily done by applying the results for one population to an invariant subset of the form $\Sp\cdots\Sp\mu_\sigma\Sp\cdots\Sp$, with $\sigma\in [M]$ chosen such that $\frac{\partial}{\partial\gamma}g_\sigma(0,\gamma)\neq 0$. No asymptotic stability in this subset with one free population then yields no asymptotic stability in the whole system with $M$ free populations.

Thus, we only consider the case $M=1$, so we assume $\frac{\partial}{\partial\gamma}g_1(0,\gamma)\neq 0$. The strategy of the proof is to construct a sequence of steady states converging to the synchronized state. Along this family, no asymptotic convergence can take place. To accomplish this construction, consider a perturbation of the synchronized state of the form 
\begin{align}\label{eq:asymptotic_initial_measure}
	\mu_1^\mathrm{in} = \left(1-\frac{1}{n}\right)\delta_{\phi_1^\mathrm{in}} + \frac{1}{n}\delta_{\phi_2^\mathrm{in}},
\end{align}
for $\phi_1^\mathrm{in},\phi_2^\mathrm{in},\in\S$ and $n\in \N$. Now, $\mu(t):=\mu_1(t)$ obeys the equations
\begin{align*}
	\partial_t\Phi(t,\xi,\mu^\mathrm{in}) &= (\mathcal K\mu(t))(\Phi(t,\xi,\mu^\mathrm{in}))\\
	\mu(t) &= \Phi(t,\cdot,\mu^\mathrm{in})\#\mu^\mathrm{in}\\
	\Phi(0,\xi,\mu^\mathrm{in}) &= \xi,
\end{align*}
with
\begin{align*}
	(\mathcal K\mu)(\phi) = \omega_1 + \int_\S \int_{\S^{\chi}}\int_{\S^{\chi}}  g_1(\alpha-\beta,\gamma-\phi) \ \d\mu^{\chi}(\alpha)\d\mu^{\chi}(\beta)\d\mu(\gamma),
\end{align*}
and $\chi = \abs{r^1}$.

As the push-forward of $\mu^\mathrm{in}$, specifically the convex combination of two dirac distributions, is again a convex combination of two diracs, the solution $\mu(t)$ is given by $\mu(t) = \left(1-\frac{1}{n}\right) \delta_{\phi_1(t)} + \frac{1}{n}\delta_{\phi_2(t)}$ for $\phi_1(t) = \Phi(t,\phi_1^\mathrm{in},\mu^\mathrm{in})$ and $\phi_2(t) = \Phi(t,\phi_2^\mathrm{in},\mu^\mathrm{in})$. Their difference $\Psi(t):=\phi_2(t)-\phi_1(t)$ satisfies the differential equation
\begin{align*}
	\dot\Psi(t) &= \dot\phi_2(t)-\dot\phi_1(t)\\
	&=\int_\S \int_{\S^\chi} \int_{\S^\chi} g_1(\alpha-\beta,\gamma-\phi_2(t))- g_1(\alpha-\beta,\gamma-\phi_1(t)) \ \d\mu^\chi(\alpha)\d\mu^\chi(\beta)\d\mu(\gamma)\\
	&=  g_1(\phi_1(t)-\phi_1(t),\phi_1(t)-\phi_2(t)) - g_1(\phi_1(t)-\phi_1(t),\phi_1(t)-\phi_1(t)) + \mathcal O\left( \frac{1}{n}\right)\\
	&= g_1(0,-\Psi(t)) - g_1(0,0) + \mathcal O\left(\frac{1}{n}\right) =: f_n(\Psi).
\end{align*}
Obviously, $\mu^\mathrm{in}\in\Sp$ if and only if $\phi_1^\mathrm{in} = \phi_2^\mathrm{in}$ and further, $\mu(t)\to\Sp$ if and only if $\snorm{\Psi}\to 0$. However, as we have assumed $\frac{\partial}{\partial \gamma}g_1(0,\gamma)\neq 0$, for each large enough $n$, there exist $\Psi^0_n$ with $\snorm{\Psi^0_n}>0$ such that $f_n(\Psi^0_n) = 0$. Consequently, if $\Psi(0)=\Psi^0_n$, $\Psi(t)=\Psi^0_n$ for all $t\ge 0$.

Now suppose for contradiction that $\Sp$ was asymptotically stable. Then, according to Definition \ref{def:MPopAsymStability}, there must exist a neighborhood $V$ of $\Sp$ such that for all $\mu^\mathrm{in}\in V$, the solution of \eqref{eq:SystemOfEquations},\eqref{eq:coupling_differences} satisfies $\mu(t)\to \Sp$ as $t\to \infty$. Then, there also has to exist $\epsilon_V>0$ with $\B(\Sp,\epsilon_V)\subset V$. On the one hand, choosing $n$ such that $\frac{\pi}{n}<\epsilon_V$ and the initial measure according to \eqref{eq:asymptotic_initial_measure} with $\phi_2^\mathrm{in}-\phi_1^\mathrm{in} = \Psi_n^0$ yields 
\begin{align*}
	W_1(\Sp,\mu^\mathrm{in})\le W_1\left(\delta_{\phi_1^\mathrm{in}}, \left(1-\frac{1}{n}\right)\delta_{\phi_1^\mathrm{in}} + \frac{1}n \delta_{\phi_2^\mathrm{in}}\right) = \frac{1}{n}\snorm{\Psi^0_n}\le \frac{\pi}{n}<\epsilon_V,
\end{align*}
so $\mu^\mathrm{in}\in V$. On the other hand $\phi_2(t)-\phi_1(t) = \Psi_n^0$ for all $t\ge 0$ and thus $\mu(t)$ does not converge to $\Sp$, which contradicts asymptotic stability.

\subsubsection{Stability}

Next, we are going to show that large classes of generic systems do admit at least stability of the synchronized state for large parameter regions. To begin, we use the following abbreviations:
\begin{align*}
	g^{(0,1)}_\sigma(\alpha,\gamma) &:= \frac{\partial}{\partial \gamma}g_\sigma(\alpha,\gamma)\\
	a_\sigma &= g_\sigma^{(0,1)}(0,0).
\end{align*}
Let us now assume that $a_\sigma>0 \text{ for all }\sigma \in [M]$ and 
\begin{align}\label{eq:kappa_assumption}
	a_\sigma-\kappa<g_\sigma^{(0,1)}(\alpha,\phi)
\end{align}
for all $\sigma\in [M]$ and all $\alpha\in (-\eta,\eta)^{\abs{r^\sigma}}, \phi\in(-\eta,\eta)$. We will later impose conditions on $\kappa>0$ and then choose $\eta>0$ accordingly to \eqref{eq:kappa_assumption}. We want to remark that results from this section rely on the notation defined in Section~\ref{sec:introduction}.

\begin{lemma}\label{lem:constant_mass}
	Let $\xi_1,\xi_2\in \S$. For any $\sigma\in[M], \mu^\mathrm{in}\in\mathcal P(\S)^M$ and $\phi_1(t) := \Phi_\sigma(t,\xi_1,\mu^\mathrm{in})$ and $\phi_2(t):=\Phi_\sigma(t,\xi_2,\mu^\mathrm{in})$, the mass of $\mu_\sigma(t)$ inside $(\phi_1(t),\phi_2(t))$,  i.e.,
	\begin{align*}
	    \int_{(\phi_1(t),\phi_2(t))} \ \mu_\sigma(t,\d\alpha)
	\end{align*}
	remains constant over time.
\end{lemma}

\begin{proof}
	By continuity of $\Phi_\sigma(t,\xi,\mu^\mathrm{in})$ with respect to $t$ and $\xi$, $\Phi^{-1}_\sigma(t,(\phi_1(t),\phi_2(t)),\mu^\mathrm{in}) = (\phi_1(0),\phi_2(0)) = (\xi_1,\xi_2)$, where the inverse is taken only with respect to the $\xi$ variable. By \eqref{eq:SystemOfEquations_PushForward},
	\begin{align*}
		\int_{(\phi_1(t),\phi_2(t))} \ \mu_\sigma(t,\d\alpha) = \int_{\Phi_\sigma^{-1}(t,(\phi_1(t),\phi_2(t)),\mu^\mathrm{in})}\ \mu^\mathrm{in}_\sigma(\d\alpha) = \int_{(\xi_1,\xi_2)}\ \mu_\sigma^\mathrm{in}(\d\alpha),
	\end{align*}
	for all $t\ge 0$. Therefore, the integral on the left-hand side in indeed independent of $t$.
\end{proof}

\begin{lemma}\label{lem:Psi_differential}
	For all $\sigma\in[M]$ and any two particles $\phi_1^\sigma(t) := \Phi_\sigma(t,\xi_1^\sigma,\mu^\mathrm{in})$ and $\phi_2^\sigma(t):=\Phi_\sigma(t,\xi_2^\sigma,\mu^\mathrm{in})$, we define the phase difference $\Psi_\sigma(t) := \phi_2^\sigma(t)-\phi_1^\sigma(t)\in [0,2\pi)$. Let $m^\mathrm{inside}_\sigma$ denote the $\mu_\sigma$-mass inside the interval $(\phi_1^\sigma(t),\phi_2^\sigma(t))$, which is by Lemma \ref{lem:constant_mass} independent of $t$. Then, there exists a constant $C>0$ such that whenever $0<\Psi_\sigma(t)<\eta$ for all $\sigma\in[M]$, they satisfy
	\begin{align*}
		\dot\Psi_\sigma(t) < -\Psi_\sigma(t)(a_\sigma-\kappa) \left(\min_{i\in[M]} m^\mathrm{inside}_i\right)^{2\abs{r^\sigma}+1} + C (1-m^\mathrm{inside}_\sigma).
	\end{align*}
\end{lemma}

\begin{proof}
	Let us consider a decomposition of the probability measures $\mu_\sigma(t)$ into two measures $\mu_\sigma^\textrm{inside}(t)$ and $\mu_\sigma^\textrm{outside}(t)$ with
	\begin{align}\label{eq:mu_inside_outside}
		\mu_\sigma(t) = \mu_\sigma^\textrm{inside}(t) + \mu_\sigma^\textrm{outside}(t)
	\end{align}
	and $\operatorname{supp}(\mu_\sigma^\textrm{inside}(t))\subset ((\phi_1^\sigma(t), \phi_2^\sigma(t))$, $\operatorname{supp}(\mu_\sigma^\textrm{outside}(t))\subset  \S\setminus((\phi_1^\sigma(t), \phi_2^\sigma(t))$. Then, a calculation shows
	{\allowdisplaybreaks
	\begin{align*}
		\dot\Psi_\sigma(t) &= \dot \phi_2^\sigma(t)-\dot\phi_1^\sigma(t)\\
		&= (\mathcal K_\sigma\mu(t))(\phi_2^\sigma(t))-(\mathcal K_\sigma\mu(t))(\phi_1^\sigma(t))\\
		&= \int_\S\int_{\S^{\abs{r^\sigma}}}\int_{\S^{\abs{r^\sigma}}}g_\sigma(\alpha-\beta,\gamma-\phi_2^\sigma(t))-g_\sigma(\alpha-\beta,\gamma-\phi_1^\sigma(t))\ \d\mu^{(r^\sigma)}(\alpha)\d\mu^{(r^\sigma)}(\beta)\d\mu_\sigma(\gamma)\\
		&\stackrel{(*)}{=} \int_\S\int_{\S^{\abs{r^\sigma}}}\int_{\S^{\abs{r^\sigma}}}g_\sigma(\alpha-\beta,\gamma-\phi_2^\sigma(t))-g_\sigma(\alpha-\beta,\gamma-\phi_1^\sigma(t))\\
		&\qquad \d\mu^{\textrm{inside}^{(r^\sigma)}}(\alpha)\d\mu^{\textrm{inside}^{(r^\sigma)}}(\beta)\d\mu_\sigma^\textrm{inside}(\gamma) + \text{integrals over }\mu^\textrm{outside}\\
		&\stackrel{(**)}{<} \int_\S\int_{\S^{\abs{r^\sigma}}}\int_{\S^{\abs{r^\sigma}}} \Big[g_\sigma(\alpha-\beta,0)+(\gamma-\phi_2^\sigma(t))(a_\sigma-\kappa)\\
		&\qquad - (g_\sigma(\alpha-\beta,0) + (\gamma-\phi_1^\sigma(t))(a_\sigma-\kappa))\Big] \ \d\mu^{\textrm{inside}^{(r^\sigma)}}(\alpha)\d\mu^{\textrm{inside}^{(r^\sigma)}}(\beta)\d\mu_\sigma^\textrm{inside}(\gamma)\\
		&\qquad + \text{integrals over }\mu^\textrm{outside}\\
		&=(a_\sigma-\kappa) \int_\S\int_{\S^{\abs{r^\sigma}}}\int_{\S^{\abs{r^\sigma}}} -\Psi_\sigma(t) \ \d\mu^{\textrm{inside}^{(r^\sigma)}}(\alpha)\d\mu^{\textrm{inside}^{(r^\sigma)}}(\beta)\d\mu_\sigma^\textrm{inside}(\gamma)\\
		&\qquad + \text{integrals over }\mu^\textrm{outside}\\
		&= -\Psi_\sigma(t)(a_\sigma-\kappa) \left(\prod_{i=1}^{\abs{r^\sigma}} m^\mathrm{inside}_{(r^\sigma_i)} \right)^2 m^\mathrm{inside}_\sigma + \text{integrals over }\mu^\textrm{outside}\\
		&< -\Psi_\sigma(t)(a_\sigma-\kappa) \left(\min_{i\in[M]} m^\mathrm{inside}_i\right)^{2\abs{r^\sigma}+1} + C (1-m^\mathrm{inside}_\sigma).
	\end{align*}}%
    Here, the equality $(*)$ was achieved by decomposing each measure $\mu_i$ into its components according to~\eqref{eq:mu_inside_outside} and rearranging terms such that every integrand with an integral running over at least one measure of the type $\mu^\textrm{outside}$ is contained in the part  ``integrals over $\mu^\textrm{outside}$''. We can easily estimate these integrals from above by first combining the integrals into a single one running over $\mu_\sigma^\textrm{outside}$ (with the integrand still consisting of integrals) and then taking the supremum norm $C$ of the integrand. As the total mass of $\mu_\sigma(t)$ equals $1$ and the $\mu_\sigma$-mass inside the interval $(\phi_1^\sigma(t),\phi_2^\sigma(t))$ is $m^\mathrm{inside}_\sigma$, the mass outside this interval evaluates to $1-m^\mathrm{inside}_\sigma$. Consequently, the terms summarized in ``integrals over $\mu^\textrm{outside}$'' can be bounded from above by $C(1-m^\mathrm{inside}_\sigma)$. The inequality $(**)$ is based on linear approximation and the fact that $\Psi_\sigma(t)<\eta$.
\end{proof}

\begin{lemma}\label{lem:WassersteinClosenessCircle}
	Let $\mu\in \mathcal P(\S)$ be a probability measure on the circle, $\xi_1,\xi_2\in\S$ and
	\begin{align*}
		m^\mathrm{inside} = \int_{(\xi_1,\xi_2)} \d \mu(\alpha).
	\end{align*}
	Then, $W_1(\Sp, \mu) < (\xi_2-\xi_1)m^\mathrm{inside} + \pi (1-m^\mathrm{inside})$.
\end{lemma}
\begin{proof}
	Using \eqref{eq:WassersteinCoupling}, we calculate
	\begin{align*}
		W_1(\Sp, \mu) &= \inf_{\delta\in \Sp} W_1(\delta, \mu)\\
		& \le W_1(\delta_{\xi_1}, \mu)\\
		& \le \int_{\S\times\S} \snorm{\alpha-\beta} \ \gamma(\d \alpha, \d \beta),\quad \text{with } \gamma(\d \alpha,\d \beta) = \delta_{\xi_1}(\d \alpha)\mu(\d \beta)\\
		&= \int_\S \snorm{\xi_1-\beta} \ \mu(\d\beta)\\
		&= \int_{(\xi_1,\xi_2)} \snorm{\xi_1-\beta} \ \mu(\d\beta) + \int_{\S\setminus(\xi_1,\xi_2)} \snorm{\xi_1-\beta} \ \mu(\d\beta)\\
		&< (\xi_2-\xi_1) m^\mathrm{inside} + \pi (1-m^\mathrm{inside}).
	\end{align*}
\end{proof}

We can not put the previous lemmas together and formulate our main theorem:
\begin{theorem}\label{thm:sync_stability}
	If the coupling functions $g_\sigma$ are continuously differentiable, i.e. $g_\sigma \in C^1(\S^{\abs{r^\sigma}}\times \S)$, and they satisfy $a_\sigma>0$ for all $\sigma\in[M]$ then, the set of all-synchronized states $\Sp^M$ is stable.
\end{theorem}

\begin{proof}
	To verify Definition \ref{def:MPopStability}, let $U_1,\dots,U_M$ be neighborhoods of $\Sp$ and choose $\epsilon_U$ such that $\B(\Sp,\epsilon_U)\subset U_\sigma$ for all $\sigma\in[M]$. Further, let $\eta>0$ be such that \eqref{eq:kappa_assumption} is fulfilled with $\kappa = \min_{\sigma\in[M]} a_\sigma/2=:a_0/2$. Now, first choose $\zeta >0$ with $\zeta < \min(\frac{\eta}{2},\frac{\epsilon_U}{4})$ and then $\epsilon_V>0$ so small, that both
	\begin{align}\label{eq:assumption_epsV}
		\zeta a_0 \left(1-\frac{\epsilon_V}{\zeta}\right)^{2\max_{j\in [M]}\abs{r^j}+1} > C\frac{\epsilon_V}{\zeta},
	\end{align}
	with $C$ coming from Lemma \ref{lem:Psi_differential}, and $\epsilon_V < \frac{\epsilon_U\zeta}{2\pi}$. To satisfy Definition \ref{def:MPopStability} we can then take $V_\sigma = \B(\Sp,\epsilon_V)$ for all $\sigma\in [M]$.
	
	To see that indeed $\mu_\sigma(t)\in \B(\Sp,\epsilon_U)\subset U_\sigma$ for all $t\ge 0$ provided that $\mu^\mathrm{in}_\sigma \in V_\sigma$, we take $\mu^\mathrm{in}_\sigma\in V_\sigma$ and $\xi_1,\dots,\xi_M\in\S$ with $W_1(\delta_{\xi_\sigma}, \mu^\mathrm{in}_\sigma)<\epsilon_V$. The representation of the Wasserstein-$1$ distance \eqref{eq:WassersteinCoupling} yields
	\begin{align}\label{eq:Wasserstein_estimate}
		\int_{\S\setminus (\xi_\sigma-\zeta,\xi_\sigma+\zeta)} \ \d\mu^\mathrm{in}_\sigma < \frac{\epsilon_V}{\zeta}.
	\end{align}
	To see this, note that
	\begin{align*}
		W_1(\delta_{\xi_\sigma}, \mu^\mathrm{in}_\sigma) &= \int_\S \snorm{\xi_\sigma-\beta}\ \mu^\mathrm{in}_\sigma(\d\beta)\\
		&=\int_{\S\setminus (\xi_\sigma-\zeta,\xi_\sigma+\zeta)} \snorm{\xi_\sigma-\beta} \ \mu^\mathrm{in}_\sigma(\d\beta) + \int_{(\xi_\sigma-\zeta,\xi_\sigma+\zeta)} \snorm{\xi_\sigma-\beta} \ \mu^\mathrm{in}_\sigma(\d\beta)\\
		&\ge \zeta \int_{\S\setminus (\xi_\sigma-\zeta,\xi_\sigma+\zeta)}  \mu^\mathrm{in}_\sigma(\d\beta).
	\end{align*}	
	Dividing by $\zeta$ yields \eqref{eq:Wasserstein_estimate}. As a result,
	\begin{align*}
	 	m^\mathrm{inside}_\sigma := \int_{ (\xi_\sigma-\zeta,\xi_\sigma+\zeta) }  \ \d\mu^\mathrm{in}_\sigma > 1-\frac{\epsilon_V}{\zeta}.
	\end{align*}
	If we now trace the $2M$ particles defined by $\phi_1^\sigma(t) := \Phi_\sigma(t,\xi_\sigma-\zeta, \mu^\mathrm{in})$ and $\phi_2^\sigma(t):=\Phi_\sigma(t,\xi_\sigma+\zeta,\mu^\mathrm{in})$, Lemma \ref{lem:constant_mass} yields
	\begin{align*}
		\int_{\S\setminus (\phi_1^\sigma(t),\phi_2^\sigma(t))} \mu_\sigma(t,\d\gamma) < \frac{\epsilon_V}{\zeta}
	\end{align*}
	and
	\begin{align*}
		\int_{(\phi_1^\sigma(t),\phi_2^\sigma(t))} \mu_\sigma(t,\d\gamma) > 1-\frac{\epsilon_V}{\zeta},
	\end{align*}
	for all $t\ge 0$. Next, we apply Lemma \ref{lem:Psi_differential} to obtain that the phase differences $\Psi_\sigma(t) := \phi_2^\sigma(t)-\phi_1^\sigma(t)$ satisfy
	\begin{align}
		\label{eq:psi_flow}
		\dot\Psi_\sigma(t) &< -\Psi_\sigma(t)(a_\sigma-\kappa)\left(\min_{i\in[M]}m^\mathrm{inside}_i \right)^{2\abs{r^\sigma}+1} + C(1-m^\mathrm{in}_\sigma)\\
		\nonumber
		&\le -\Psi_\sigma(t)\frac{a_0}{2}\left(\min_{i\in[M]}m^\mathrm{inside}_i \right)^{2\abs{r^\sigma}+1} + C\frac{\epsilon_V}{\zeta}\\
		\label{eq:tobenegative}
		&\le -\Psi_\sigma(t)\frac{a_0}{2}\left(\min_{i\in[M]}m^\mathrm{inside}_i \right)^{2\max_{j\in[M]}\abs{r^j}+1} + C\frac{\epsilon_V}{\zeta},
	\end{align}
	which stays valid if $\Psi_\sigma(t)<\eta$ for all $\sigma \in [M]$. First note, that the choice of $\epsilon_V$ such that \eqref{eq:assumption_epsV} holds true, yields that for $\Psi_\sigma(t) = 2\zeta$, the right-hand side of \eqref{eq:tobenegative} is negative:
	\begin{align}\label{eq:negative_calc}
		-2\zeta\frac{a_0}{2}\left(\min_{i\in[M]}m^\mathrm{inside}_i \right)^{2\max_{j\in[M]}\abs{r^j}+1}+C\frac{\epsilon_V}{\zeta} &< -\zeta a_0\left(1-\frac{\epsilon_V}{\zeta} \right)^{2\max_{j\in[M]}\abs{r^j}+1}+C\frac{\epsilon_V}{\zeta} <0.
	\end{align}
	Therefore, for $\Psi_\sigma(t) = 2\zeta$, the derivative satisfies $\dot\Psi_\sigma(t)<0$ if all other components satisfy $\Psi_i(t)<\eta$ for all $i\neq\sigma$. To be precise, the region 
	\begin{align*}
		\mathcal R := \{ \Psi \in \R^n: \Psi_\sigma\in [0, \Psi_\sigma(0)]\}
	\end{align*}
	is invariant under the flow of \eqref{eq:psi_flow}. To see that, first note that $\Psi_\sigma(0) = 2\zeta$ for all $\sigma = 1,\dots,M$, so~$\mathcal R$ is effectively a hyper cube $\mathcal R = [0,2\zeta]^M$. For given $\sigma$, the component $\Psi_\sigma(t)$ can not leave the hyper cube through $0$, because that would mean that the two particles $\phi_1^\sigma(t)$ and $\phi_2^\sigma(t)$ collide. A trajectory $(\Psi_1(t), \dots, \Psi_M(t))$ also can not leave $\mathcal R$ by one component exceeding the value $2\zeta$. Suppose, for a contradiction that there is a time $t^\star>0$ such that $\Psi_\sigma(t^\star) = 2\zeta$ for one $\sigma\in [M]$ and let $t^\star$ be the first time that happens. Then, however, all components still satisfy $\Psi_\sigma(t^\star) \le 2\zeta < \eta$  and thus \eqref{eq:tobenegative} is valid. By the calculation \eqref{eq:negative_calc} $\dot \Psi_\sigma(t^\star)<0$, so $\mathcal R$ is indeed invariant.
	
	Consequently, for all $\sigma = 1,\dots,M$ and all $t\ge 0$, we have $\Psi_\sigma(t)< 2\zeta$ and thus, by Lemma \ref{lem:WassersteinClosenessCircle},
	\begin{align*}
		W_1(\Sp,\mu_\sigma(t)) &< \Psi(t) m^\mathrm{inside}_\sigma + \pi (1-m^\mathrm{inside}_\sigma)\\
		&<2\zeta+\pi\frac{\epsilon_V}{\zeta}\\
		&<\frac{\epsilon_U}{2}+\frac{\epsilon_U}{2} = \epsilon_U.
	\end{align*}
	So indeed $\mu_\sigma(t)\in \B(\Sp,\epsilon_U)\subset U_\sigma$ for all $t\ge 0$. This verifies Definition \ref{def:MPopStability} and therefore concludes the proof.
\end{proof}

\subsubsection{Almost Asymptotic Stability}

One might now hope that although we do not have asymptotic stability, we can expect asymptotic stability of large classes of initial conditions as the family of steady states constructed above is a rather small part of phase space.

Before stating with theorems regarding asymptotic stability, we need to introduce the concept of phase differences, as this concept becomes important in the subsequent proofs. Similarly to the original system~\eqref{eq:SystemOfEquations}, the system of phase differences describes the temporal evolution of oscillators, which can be grouped into populations, on the circle. Unlike the original system~\eqref{eq:SystemOfEquations}, in the system of phase differences, the position of the oscillators is not given in absolute coordinates but instead with respect to reference oscillators. The system of phase differences is given by
\begin{subequations}\label{eq:phase_differences}
\begin{align}
	\label{eq:phase_differences_derivative}
	\partial_t\Psi_\sigma(t,\xi^\mathrm{in}_\sigma,\nu^\mathrm{in}) &= (\mathcal F_\sigma\nu(t))(\Psi_\sigma(t,\xi^\mathrm{in}_\sigma,\nu^\mathrm{in})),\\
	\label{eq:phase_differences_pushforward}
	\nu_\sigma(t)&=\Psi_\sigma(t,\cdot,\nu^\mathrm{in})\#\nu^\mathrm{in}_\sigma,\\
	\label{eq:phase_differences_initial}
	\Psi_\sigma(0,\xi^\mathrm{in}_\sigma,\nu^\mathrm{in}) &= \xi^\mathrm{in}_\sigma,
\end{align}
\end{subequations}
with
\begin{align}\label{eq:coupling_phase_differences}
	(\mathcal F_\sigma\nu)(\psi) = \int_\S\int_{\S^{\abs{r^\sigma}}}\int_{\S^{\abs{r^\sigma}}} g_\sigma(\alpha-\beta,\gamma-\psi)-g_\sigma(\alpha-\beta,\gamma)\ \d\nu^{(r^\sigma)}(\alpha)\d\nu^{(r^\sigma)}(\beta)\d\nu_\sigma(\gamma).
\end{align}

\begin{notation}
	Let $\zeta\in\S$. When using the notation $m_\zeta$, we refer to the function $m_\zeta\colon \S\to\S$ with $m_\zeta(\phi) = \phi-\zeta$.
\end{notation}

\begin{lemma}\label{lem:phase_differences}
	Let $\zeta_1,\dots,\zeta_M\in\S, \mu^\mathrm{in}\in\mathcal P(\S)^M$, suppose that $\mu(t)$ solves the system \eqref{eq:SystemOfEquations}, \eqref{eq:coupling_differences} and let $\Phi_\sigma(t,\xi^\mathrm{in}_\sigma,\mu^\mathrm{in})$ be its mean-field characteristic flow. Now define
	\begin{align}\label{eq:nu_def}
		\nu_\sigma(t):= m_{\Phi_\sigma(t,\zeta_\sigma,\mu^\mathrm{in})} \# \mu_\sigma(t) ,\qquad \nu_\sigma^\mathrm{in}:=\nu_\sigma(0)
	\end{align}
	and
	\begin{align*}
		\Psi_\sigma(t,\xi^\mathrm{in}_\sigma,\nu^\mathrm{in}):=\Phi_\sigma(t,\zeta_\sigma+\xi^\mathrm{in}_\sigma, \mu^\mathrm{in}) - \Phi_\sigma(t,\zeta_\sigma,\mu^\mathrm{in})
	\end{align*}
	for $\sigma \in[M]$. Then, $\nu(t)$ and $\Psi_\sigma(t,\xi^\mathrm{in}_\sigma,\nu^\mathrm{in})$ solve the system \eqref{eq:phase_differences}, \eqref{eq:coupling_phase_differences}.
\end{lemma}

Before proving this lemma, we remark that $\Phi_\sigma(t,\zeta_\sigma,\mu^\mathrm{in})$ can be seen as the position of reference oscillators we have talked about at the beginning of this section.

\begin{proof}[Proof of Lemma \ref{lem:phase_differences}]
	It is easy to verify \eqref{eq:phase_differences_initial}:
	\begin{align*}
		\Psi_\sigma(0,\xi^\mathrm{in}_\sigma,\nu^\mathrm{in}) &= \Phi_\sigma(0,\zeta_\sigma+\xi^\mathrm{in}_\sigma,\mu^\mathrm{in}) - \Phi_\sigma(0,\zeta_\sigma,\mu^\mathrm{in})\\
		&=\zeta_\sigma + \xi^\mathrm{in}_\sigma - \zeta_\sigma = \xi^\mathrm{in}_\sigma.
	\end{align*}
	To check \eqref{eq:phase_differences_pushforward}, take a measurable set $A\subset \S$ and calculate
	\begin{align*}
		(\Psi_\sigma(t,\cdot,\nu^\mathrm{in})\#\nu_\sigma^\mathrm{in})(A) & \stackrel{\eqref{eq:nu_def}}{=} (\Psi_\sigma(t,\cdot,\nu^\mathrm{in}) \# (m_{\zeta_\sigma}\#\mu_\sigma^\mathrm{in}))(A)\\
		&=(m_{\zeta_\sigma}\#\mu_\sigma^\mathrm{in})(\Psi^{-1}_\sigma(t,A,\nu^\mathrm{in}))\\
		&=\mu^\mathrm{in}_\sigma(m^{-1}_{\zeta_\sigma}(\Psi^{-1}_\sigma(t,A,\nu^\mathrm{in})))\\
		&=\mu^\mathrm{in}_\sigma(\zeta_\sigma + \Psi^{-1}_\sigma(t,A,\nu^\mathrm{in}))\\
		&=\mu^\mathrm{in}_\sigma(\zeta_\sigma + \Phi_\sigma^{-1}(t,A+\Phi_\sigma(t,\zeta_\sigma,\mu^\mathrm{in}),\mu^\mathrm{in})-\zeta_\sigma)\\
		&=\mu^\mathrm{in}_\sigma(\Phi_\sigma^{-1}(t,A+\Phi_\sigma(t,\zeta_\sigma,\mu^\mathrm{in}),\mu^\mathrm{in})\\
		&=(\Phi_\sigma(t,\cdot,\mu^\mathrm{in}) \# \mu^\mathrm{in})(A+\Phi_\sigma(t,\zeta_\sigma,\mu^\mathrm{in}))\\
		&=(m_{\Phi_\sigma(t,\zeta_\sigma,\mu^\mathrm{in})}\#\mu(t))(A)\\
		&=\nu(t)(A).
	\end{align*}
	To finally show \eqref{eq:phase_differences_derivative}, we use the notation
	\begin{align*}
		\zeta^{(r^\sigma)} &= (\zeta_{(r^\sigma_1)},\dots,\zeta_{(r^\sigma_{\abs{r^\sigma}})}),\\
		\Phi^{(r^\sigma)}(t,\alpha,\mu^\mathrm{in}) &= (\Phi_{(r^\sigma_1)}(t,\alpha_1,\mu^\mathrm{in}),\dots,\Phi_{(r^\sigma_{\abs{r^\sigma}})}(t,\alpha_{\abs{r^\sigma}}, \mu^\mathrm{in})),\\
		\Psi^{(r^\sigma)}(t,\alpha,\nu^\mathrm{in}) &= (\Psi_{(r^\sigma_1)}(t,\alpha_1,\nu^\mathrm{in}),\dots,\Psi_{(r^\sigma_{\abs{r^\sigma}})}(t,\alpha_{\abs{r^\sigma}}, \nu^\mathrm{in})).
	\end{align*}
	Then, a rather lengthy calculation confirms
	{\allowdisplaybreaks
	\begin{align*}
		\partial_t\Psi_\sigma(t,\xi^\mathrm{in}_\sigma,\nu^\mathrm{in}) &= (\mathcal K_\sigma\mu(t))(\Phi_\sigma(t,\zeta_\sigma+\xi^\mathrm{in}_\sigma,\mu^\mathrm{in})) - (\mathcal K_\sigma\mu(t))(\Phi_\sigma(t,\zeta_\sigma,\mu^\mathrm{in}))\\
		&=\omega_\sigma + \int_\S\int_{\S^{\abs{r^\sigma}}}\int_{\S^{\abs{r^\sigma}}} g_\sigma(\alpha-\beta,\gamma-\Phi_\sigma(t,\zeta_\sigma+\xi^\mathrm{in}_\sigma,\mu^\mathrm{in}))\ \mu^{(r^\sigma)}(t,\d\alpha)\mu^{(r^\sigma)}(t,\d\beta)\mu_\sigma(t,\d\gamma)\\
		&\qquad - \omega_\sigma - \int_\S\int_{\S^{\abs{r^\sigma}}}\int_{\S^{\abs{r^\sigma}}} g_\sigma(\alpha-\beta,\gamma-\Phi_\sigma(t,\zeta_\sigma,\mu^\mathrm{in}))\ \mu^{(r^\sigma)}(t,\d\alpha)\mu^{(r^\sigma)}(t,\d\beta)\mu_\sigma(t,\d\gamma)\\
		&=\int_\S\int_{\S^{\abs{r^\sigma}}}\int_{\S^{\abs{r^\sigma}}} g_\sigma\Big(\Phi^{(r^\sigma)}(t,\alpha,\mu^\mathrm{in})-\Phi^{(r^\sigma)}(t,\beta,\mu^\mathrm{in}),\\
		&\qquad\qquad \Phi_\sigma(t,\gamma,\mu^\mathrm{in})-\Phi_\sigma(t,\zeta_\sigma+\xi^\mathrm{in}_\sigma,\mu^\mathrm{in})\Big)\ \d\mu^{\textrm{in}^{(r^\sigma)}}(\alpha)\d\mu^{\textrm{in}^{(r^\sigma)}}(\beta)\d\mu^\mathrm{in}_\sigma(\gamma)\\
		&\qquad -\int_\S\int_{\S^{\abs{r^\sigma}}}\int_{\S^{\abs{r^\sigma}}} g_\sigma\Big(\Phi^{(r^\sigma)}(t,\alpha,\mu^\mathrm{in})-\Phi^{(r^\sigma)}(t,\beta,\mu^\mathrm{in}),\\
		&\qquad\qquad \Phi_\sigma(t,\gamma,\mu^\mathrm{in})-\Phi_\sigma(t,\zeta_\sigma,\mu^\mathrm{in})\Big)\ \d\mu^{\textrm{in}^{(r^\sigma)}}(\alpha)\d\mu^{\textrm{in}^{(r^\sigma)}}(\beta)\d\mu^\mathrm{in}_\sigma(\gamma)\\
		&=\int_\S\int_{\S^{\abs{r^\sigma}}}\int_{\S^{\abs{r^\sigma}}} g_\sigma\Big(\Psi^{(r^\sigma)}(t,\alpha-\zeta^{(r^\sigma)},\nu^\mathrm{in})-\Psi^{(r^\sigma)}(t,\beta-\zeta^{(r^\sigma)},\nu^\mathrm{in}),\\
		&\qquad\qquad \Psi_\sigma(t,\gamma-\zeta_\sigma,\nu^\mathrm{in})-\Psi_\sigma(t,\xi^\mathrm{in}_\sigma,\nu^\mathrm{in})\Big)\ \d\mu^{\textrm{in}^{(r^\sigma)}}(\alpha)\d\mu^{\textrm{in}^{(r^\sigma)}}(\beta)\d\mu^\mathrm{in}_\sigma(\gamma)\\
		&\qquad -\int_\S\int_{\S^{\abs{r^\sigma}}}\int_{\S^{\abs{r^\sigma}}} g_\sigma\Big(\Psi^{(r^\sigma)}(t,\alpha-\zeta^{(r^\sigma)},\nu^\mathrm{in})-\Psi^{(r^\sigma)}(t,\beta-\zeta^{(r^\sigma)},\nu^\mathrm{in}),\\
		&\qquad\qquad \Psi_\sigma(t,\gamma-\zeta_\sigma,\nu^\mathrm{in})\Big)\ \d\mu^{\textrm{in}^{(r^\sigma)}}(\alpha)\d\mu^{\textrm{in}^{(r^\sigma)}}(\beta)\d\mu^\mathrm{in}_\sigma(\gamma)\\		
		&=\int_\S\int_{\S^{\abs{r^\sigma}}}\int_{\S^{\abs{r^\sigma}}} g_\sigma\Big(\Psi^{(r^\sigma)}(t,\alpha,\nu^\mathrm{in})-\Psi^{(r^\sigma)}(t,\beta,\nu^\mathrm{in}),\\
		&\qquad\qquad \Psi_\sigma(t,\gamma,\nu^\mathrm{in})-\Psi_\sigma(t,\xi^\mathrm{in}_\sigma,\nu^\mathrm{in})\Big)\ \d\nu^{\textrm{in}^{(r^\sigma)}}(\alpha)\d\nu^{\textrm{in}^{(r^\sigma)}}(\beta)\d\nu^\mathrm{in}_\sigma(\gamma)\\
		&\qquad -\int_\S\int_{\S^{\abs{r^\sigma}}}\int_{\S^{\abs{r^\sigma}}} g_\sigma\Big(\Psi^{(r^\sigma)}(t,\alpha,\nu^\mathrm{in})-\Psi^{(r^\sigma)}(t,\beta,\nu^\mathrm{in}),\\
		&\qquad\qquad \Psi_\sigma(t,\gamma,\nu^\mathrm{in})\Big)\ \d\nu^{\textrm{in}^{(r^\sigma)}}(\alpha)\d\nu^{\textrm{in}^{(r^\sigma)}}(\beta)\d\nu^\mathrm{in}_\sigma(\gamma)\\
		&=\int_\S\int_{\S^{\abs{r^\sigma}}}\int_{\S^{\abs{r^\sigma}}} g_\sigma(\alpha-\beta, \gamma-\Psi_\sigma(t,\xi^\mathrm{in}_\sigma,\nu^\mathrm{in}))\ \nu^{(r^\sigma)}(t,\d\alpha)\nu^{(r^\sigma)}(t,\d\beta)\nu_\sigma(t,\d\gamma)\\
		&\qquad -\int_\S\int_{\S^{\abs{r^\sigma}}}\int_{\S^{\abs{r^\sigma}}} g_\sigma(\alpha-\beta,\gamma)\ \nu^{(r^\sigma)}(t,\d\alpha)\nu^{(r^\sigma)}(t,\d\beta)\nu^\mathrm{in}_\sigma(t,\d\gamma)\\
		&=(\mathcal F_\sigma\nu)(\Psi_\sigma(t,\xi^\mathrm{in}_\sigma,\nu^\mathrm{in})).
	\end{align*}
	}%
	This completes the proof.
\end{proof}

\begin{remark}\label{rem:mu_nu_difference}
	This Lemma is especially useful because the only operation used to create the measures~$\nu_\sigma(t)$ from the measures~$\mu_\sigma(t)$ is a rotation by $\Phi_\sigma(t,\zeta_\sigma,\mu^\mathrm{in})$ around the circle. Therefore, $W_1(\Sp,\nu_\sigma(t)) = W_1(\Sp,\mu_\sigma(t))$ and $\mu_\sigma(t)\to \Sp$ if and only if $\nu_\sigma(t)\to \Sp$.
\end{remark}

\begin{lemma}\label{lem:C1norm}
	If the coupling functions $g_\sigma(\alpha,\gamma)$ are continuously differentiable and the derivative $g_\sigma^{(0,1)}$ is Lipschitz continuous with constant $L_1$ then, the coupling operator $\mathcal F_\sigma$ satisfies
	\begin{align*}
		(\mathcal F_\sigma\delta_0^M)(\psi) &= g_\sigma(0,-\psi)-g_\sigma(0,0),\\
		\left\lvert (\mathcal F_\sigma\delta_0^M)(\psi) - (\mathcal F_\sigma\nu)(\psi)\right\rvert &\le 4L\left(\sum_{i=1}^{\abs{r^\sigma}} W_1(\delta_0,\nu_{(r^\sigma_i)})\right) + 2LW_1(\delta_0, \nu_\sigma),\\
		\left\lvert \frac{\partial}{\partial \psi}(\mathcal F_\sigma\delta_0^M)(\psi) - \frac{\partial}{\partial \psi}(\mathcal F_\sigma\nu)(\psi)\right\rvert &\le  4L_1\left(\sum_{i=1}^{\abs{r^\sigma}} W_1(\delta_0,\nu_{(r^\sigma_i)})\right) + 2L_1W_1(\delta_0, \nu_\sigma).
	\end{align*}
\end{lemma}

\begin{proof}
	Follows from \eqref{eq:coupling_phase_differences} and Lemma \ref{lem:HighDimWasserstein}.
\end{proof}

\begin{theorem}\label{thm:sync_asymp_stability}
	Suppose that $a_\sigma>0$ for all $\sigma\in[M]$ and let the coupling functions $g_\sigma$ be chosen such that $g^{(0,1)}_\sigma$ are Lipschitz continuous with constant $L_1$. Further, assume that each of the functions
	\begin{align*}
		\psi\mapsto \hat g_\sigma(\psi) := g_\sigma(0,\psi)-g_\sigma(0,0)
	\end{align*}
	has exactly two zeros around the circle, the trivial one at $0$ and another one at $\psi^0_\sigma\in \S\setminus\{0\}$. Moreover, suppose $b_\sigma := \hat g'(\psi^0_\sigma) \neq 0$ for all $\sigma\in [M]$. Then, initial configurations in the space of densities $\mu^\mathrm{in}\in \mathcal P_\mathrm{ac}(\S)^M$, which are close enough to the all-synchronized state, converge to the all-synchronized state as $t\to \infty$.
\end{theorem}


\begin{remark}
Note that the assumption on absolute continuity of the measures eliminates the counterexamples from Section \ref{sec:NoGenericAsymptoticStability} as only perturbations into measures with densities are allowed. This clearly shows, why studying the mean-field Vlasov--Fokker--Planck equation for densities is often easier in comparison to our goal of deriving directly the maximum information from the characteristic system.
\end{remark}

\begin{proof}[Proof of Theorem \ref{thm:sync_asymp_stability}]
	As the assumptions of this theorem include the assumptions of Theorem \ref{thm:sync_stability}, we can choose $\epsilon_V>0$ such that for any $\mu^\mathrm{in}\in\B(\Sp,\epsilon_V)^M$, $\mu(t)\in\B(\Sp,\epsilon_U)^M$ for all $t\ge 0$. To prove asymptotic stability in the space of absolutely continuous measures, let $\mu^\mathrm{in}\in \B(\Sp,\epsilon_V)^M\cap\mathcal P_\mathrm{ac}(\S)^M$, with~$\epsilon_V$ specified later, and proceed analogously to the proof of Theorem \ref{thm:sync_stability} until we get to the point when $\phi_2^\sigma(t)-\phi_1^\sigma(t) \le 2\zeta$ for all $\sigma\in[M]$ and $t\ge 0$. Next, we use Lemma \ref{lem:phase_differences} in order to switch to the system of phase differences with $\zeta_\sigma = \phi_1^{\sigma,\mathrm{in}} = \phi^\sigma_1(0)$. The reference oscillators in this system of phase differences are consequently given by $\phi^\sigma_1(t)$.
	Since
	\begin{align*}
		\Phi_\sigma(t,\phi^\sigma_1(0), \mu^\mathrm{in}) = \phi_1^\sigma(t)
	\end{align*}
	for all $t\ge 0$,
	{\allowdisplaybreaks
	\begin{align*}
		\int_{(0,2\zeta)} \nu_\sigma(t,\d\gamma) &\ge \int_{(0,\phi^\sigma_2(t)-\phi^\sigma_1(t))} \nu_\sigma(t,\d\gamma)\\
		&=\int_{(0,\phi^\sigma_2(t)-\phi^\sigma_1(t))}\ (m_{\Phi_\sigma(t,\phi_1^\sigma(0), \mu^\mathrm{in})}\#\mu_\sigma(t))(\d\gamma)\\
		&=\int_{(0,\phi^\sigma_2(t)-\phi^\sigma_1(t))}\ (m_{\phi_1^\sigma(t)}\#\mu_\sigma(t))(\d\gamma)\\
		&=\int_{(\phi_1^\sigma(t), \phi_2^\sigma(t))} \ \mu_\sigma(t,\d\gamma)\\
		&= m^\mathrm{inside}_\sigma\\
		&> 1-\frac{\epsilon_V}{\zeta}.
	\end{align*}
	}%
	Therefore, a computation similar to the one in the proof of Theorem \ref{thm:sync_stability} shows
	\begin{align}\label{eq:delta0_nu_close}
		W_1(\delta_0,\nu_\sigma(t)) &\le 2\zeta + \pi \frac{\epsilon_V}{\zeta}<\epsilon_U,
	\end{align}
	so the solution $\nu(t)$ always stays close to the all-synchronized state located at the origin. In the system of phase differences, individual particles then follow the flow
	\begin{align}\label{eq:ParticleFlowPsi}
		\partial_t\Psi_\sigma(t,\xi,\nu^\mathrm{in}) = (\mathcal F_\sigma\nu(t))(\Psi_\sigma(t,\xi,\nu^\mathrm{in})).
	\end{align}
	 However, by \eqref{eq:delta0_nu_close} and Lemma \ref{lem:C1norm}, \eqref{eq:ParticleFlowPsi} can be rewritten in the form
	\begin{align}\label{eq:PerturbedFlow}
		\dot\Psi_\sigma(t) = (\mathcal F_\sigma \delta_0^M)(\Psi_\sigma(t)) + p_\sigma(\Psi_\sigma(t),t),\qquad \Psi_\sigma(0)=\Psi_\sigma^\mathrm{in}
	\end{align}
	for a small perturbation $p_\sigma\in C^1(\S\times\R)$.
	Next, we fix $\epsilon_U$ and with that also $\epsilon_V$ such that for all $\nu_1,\dots,\nu_M\in \B(\delta_0, \epsilon_U)$, $\norm{p_\sigma(\cdot,t)}_{C^1}$ is small enough such that the flow induced by the dynamical system \eqref{eq:ParticleFlowPsi} is equivalent to the one induced by $\dot\Psi_\sigma = (\mathcal F_\sigma\delta_0^M)(\Psi_\sigma)$ for all $\sigma \in[M]$. So we consider \eqref{eq:PerturbedFlow} as a perturbed one-dimensional autonomous ODE. The existence of such an $\epsilon_U$ is guaranteed by Lemma~\ref{lem:C1norm}. Specifically, this lemma states that 
	\begin{align*}
			\norm{p_\sigma(\cdot,t)}_{C^1} \le (4L\abs{r^\sigma}+2L+4L_1\abs{r^\sigma}+2L_1)\epsilon_U =: \epsilon_f^\sigma,
	\end{align*}	
	uniformly in $t$.
	A particular choice of $\epsilon_U$ can be constructed as follows: As $\hat g_\sigma'(0) = a_\sigma>0$ and there are only two roots with non-vanishing derivative on the circle, $b_\sigma<0$. Further, let us write $\eta^\sigma_1,\eta^\sigma_2>0$ for radii of intervals such that $\inf_{\alpha\in (-\eta_1^\sigma,\eta_1^\sigma)} \hat g_\sigma'(\alpha) > \frac{a_\sigma}{2}$ and $\sup_{\alpha \in (\psi^0_\sigma-\eta_2^\sigma, \phi^0_\sigma+\eta_2^\sigma)} \hat g_\sigma'(\alpha) < \frac{b_\sigma}{2}$. Now, $\epsilon_U$ can be chosen such that for all $f_\sigma\in C^1(\S)$ with $\norm{f_\sigma}_{C^1}< \epsilon_f^\sigma$ the following criteria are satisfied:
	\begin{enumerate}[(C1)]
		\item\label{item:zeros_outside} $\max_{\alpha\in \S\setminus [(-\eta_1^\sigma,\eta_1^\sigma)\cup(-\psi^0_\sigma-\eta_2^\sigma, -\psi^0_\sigma+\eta_2^\sigma)]} |f_\sigma(\alpha)| < \frac{1}{2} \min_{\alpha\in \S\setminus [(-\eta_1^\sigma,\eta_1^\sigma)\cup(-\psi^0_\sigma-\eta_2^\sigma, -\psi^0_\sigma+\eta_2^\sigma)]} |\mathcal F_\sigma\delta_0^M(\alpha)|$,
		\item\label{item:zeros_inside} $\max_{\alpha\in (-\eta_1^\sigma, \eta_1^\sigma)} |f_\sigma'(\alpha)| < \frac{1}{2}a_\sigma$ and $\max_{\alpha\in (-\psi^0_\sigma-\eta_2^\sigma, -\psi^0_\sigma+\eta_2^\sigma)} |f_\sigma'(\alpha)| < \frac{-b_\sigma}{2}$.
	\end{enumerate}
	While~(C\ref{item:zeros_outside}) ensures that $\mathcal F_\sigma\nu(t)$ has no zeros away from the roots $-\psi_0^\sigma$ and~$0$ for all $t\ge 0$, (C\ref{item:zeros_inside}) guarantees that~$\mathcal F_\sigma\nu(t)$ is strictly monotonic in the two neighborhoods around the roots. This monotonicity also causes the existence of at most one zero of~$\mathcal F_\sigma\nu(t)$ near the two roots. Even though the two zeros of $\mathcal F_\sigma\nu(t)$ may be varying over time,~(C\ref{item:zeros_inside}) ensures that the flow of~\eqref{eq:ParticleFlowPsi} is still exponentially contracting in $(-\eta_1^\sigma,\eta_1^\sigma)$ and exponentially expanding in $(-\psi^0_\sigma-\eta_2^\sigma, -\psi^0_\sigma+\eta_2^\sigma)$. Thus, at least one of two distinct test particles starting in $(-\psi^0_\sigma-\eta_2^\sigma, -\psi^0_\sigma+\eta_2^\sigma)$ leaves this region and eventually ends up in $(-\eta_1^\sigma,\eta_1^\sigma)$. Since $\Psi_\sigma(t,0,\nu^\mathrm{in}) = 0$ for all $t\ge 0$ and the contracting property of $\Psi_\sigma(t,0,\nu^\mathrm{in})$ around $0$, the particle even converges to the origin at $0$. So there exists only one trajectory, which starts at an arbitrary point $\Psi^\mathrm{in}_\sigma\in \S$, that does not converge to $0$. By assumption, $\nu_\sigma^\mathrm{in}(\{\Psi^\mathrm{in}_\sigma\}) = 0$ and hence all the mass concentrates around~$0$. Therefore, $W_1(\delta_0,\nu_\sigma(t))\to 0$ as $t\to \infty$. Because this holds true for all $\sigma\in[M]$, the all-synchronized state is asymptotically stable for absolutely continuous perturbations.
	\end{proof}

\begin{remark}
	Theorems~\ref{thm:sync_stability} and~\ref{thm:sync_asymp_stability} per se only apply to perturbations in all populations. However, if we exemplarily want to analyze the stability of $\Dp\Sp\Sp\Dp$ in a network of $M=4$ populations, Proposition~\ref{prop:reducibility} allows us to reduce the system of four populations to equations describing only the evolution of population~$\#2$ and~$\#3$ while we keep population~$\#1$ and~$\#4$ fixed in splay state. Applying Theorems \ref{thm:sync_stability} and~\ref{thm:sync_asymp_stability} consequently yields criteria for the (asymptotic) stability of $\Dp\Sp\Sp\Dp$ with respect to perturbations in the second and third population.
\end{remark}

\subsection{Linear Stability of the All-Splay State}

\subsubsection{A Review of Linear Stability with Pairwise Coupling in One Population}\label{sec:ReviewSplay}

In the easiest case we consider only one population and no higher-order coupling. Then, the velocity field~\eqref{eq:coupling_differences} is given by
\begin{align*}
	(\mathcal K \mu)(\phi) = \omega + \int g(\gamma-\phi)\ \d\mu(\gamma).
\end{align*}
A well known way \cite{Strogatz2000, Strogatz1991} for analyzing stability of the splay state is to look at the mean-field (or continuity) equation
\begin{align*}
	\frac{\partial}{\partial t}\rho(t,\phi) + \frac{\partial}{\partial \phi}\left[\left( \omega + \int_\S g(\gamma-\phi)\rho(t,\gamma)\ \d\gamma\right) \rho(t,\phi)\right] = 0,
\end{align*}
which describes the evolution of the density $\rho(t,\phi)$. Next, one typically inserts the ansatz $\rho(t,\phi) = \frac{1}{2\pi} + \epsilon \eta(t,\phi)$ into the continuity equation and collects terms of order $\epsilon$. Assuming Fourier representations
\begin{align*}
	\eta(\phi,t) = \sum_{k=1}^{\infty}c_k(t) e^{ik\phi} + c.c.,\qquad g(\gamma) = \sum_{l=1}^{\infty}a_l e^{il\gamma} + c.c.,
\end{align*}
where $c.c.$ denotes the complex conjugate of the previous term, one can derive differential equations for the evolution of the coefficients $c_k(t)$:
\begin{align*}
	c_k'(t) = -(\bar a_k + \omega) ik c_k(t),\quad k = 1,2,\dots
\end{align*}
Fortunately, these equations are uncoupled and linear stability of the splay state can thus simply be infered if $\Im(a_k)>0$ for all $k\ge 1$. In other words, when writing the coupling function $g(\gamma)$ as linear combinations of $\sin(\gamma)$, $\cos(\gamma)$ and trigonometric monomes of higher order, the prefactors of $\sin(\gamma), \sin(2\gamma),\dots$ have to be negative. A similar analysis yields the linear instability of the splay state if $\Im(a_k) <0 $ for at least one $k$.\\

\subsubsection{Non-Pairwise Coupling in Multi-Population Systems}\label{sec:SplayStability}

Let us now consider the more general case of multi-population systems and higher-order interactions. Assuming the initial measures to be represented by densities $\rho_\sigma^\mathrm{in}(\phi)$, the velocity field is given by
\begin{align*}
	V_\sigma[\rho](\phi,t) = \omega_\sigma + \int_\S \int_{\S^{\abs{r^\sigma}}}\int_{\S^{\abs{r^\sigma}}}g_\sigma(\alpha-\beta,\gamma-\phi)\ \rho^{(r^\sigma)}(t,\alpha)\ \rho^{(r^\sigma)}(t,\beta)\  \rho_\sigma(t,\gamma) \ \d\alpha\d\beta\d\gamma
\end{align*}
and the densities $\rho_\sigma(t,\phi)$ solve the continuity equation \eqref{eq:initialvalue}. Here, $\rho^{(r^\sigma)}$ is the shorthand notation for
\begin{align*}
	\rho^{(r^\sigma)}(t,\alpha) := \prod_{i = 1}^{\abs{r^\sigma}} \rho_{(r^\sigma_i)}(t,\alpha_i).
\end{align*}

In this section, we extend the formal calculations from Section \ref{sec:ReviewSplay} to the case of multi-population systems and higher-order interactions. Such a formal derivation of criteria for linear stability of the all-splay state can be done with the same techniques as those used in Section \ref{sec:ReviewSplay}. As in this section, we therefore consider a small perturbation around the all-splay state, i.e.,
\begin{align}\label{eq:alldesync_perturbation}
	\rho_\sigma(t,\phi) = \frac{1}{2\pi} + \epsilon\eta_\sigma(t,\phi),
\end{align}
with Fourier decompositions
\begin{align}\label{eq:eta_expansion}
	\eta_\sigma(t,\phi) = \sum_{k=1}^{\infty}c_k^\sigma(t)e^{ik\phi} + c.c.
\end{align}
Further, the coupling functions $g_\sigma\colon \S^{\abs{r^\sigma}}\times \S\to \R$ are supposed to be given in terms of its Fourier expansion as well:
\begin{align}
	\label{eq:g_sigma_expansion}
 	&g_\sigma(\alpha,\beta) = \sum_{b\in \Z^{\abs{r^\sigma}}} \sum_{l=0}^{\infty}a_{b,l}^\sigma e^{i\langle\alpha,b\rangle} e^{i\beta l} + c.c., \qquad a^\sigma_{\mathbf 0, 0}=0, \quad \mathbf 0 = 0^{\abs{r^\sigma}}\\
	\nonumber 	
 	&\langle\alpha, b\rangle = \sum_{i = 1}^{\abs{r^\sigma}} \alpha_i b_i.
\end{align}
The requirement $a^\sigma_{\mathbf 0, 0}=0$ is not really a limitation as possible non-zero values of $a^\sigma_{\mathbf 0, 0}=0$ can be absorbed into $\omega_\sigma$.

Given these representations, we formally insert \eqref{eq:alldesync_perturbation} into the continuity equation \eqref{eq:initialvalue} to obtain
\begin{align*}
	&\frac{\partial}{\partial t}\left( \frac{1}{2\pi}+\epsilon \eta_\sigma(t,\phi)\right) + \frac{\partial}{\partial \phi}\Bigg[ \left(\frac{1}{2\pi}+\epsilon \eta_\sigma(t,\phi)\right)\Bigg( \omega_\sigma + \int_\S\int_{\S^{\abs{r^\sigma}}}\int_{\S^{\abs{r^\sigma}}}\bigg\lbrace g_\sigma(\alpha-\beta,\gamma-\phi)\\
	&\qquad  \cdot \left(\frac{1}{2\pi}+\epsilon\eta(t,\alpha)\right)^{(r^\sigma)}\left(\frac{1}{2\pi}+\epsilon\eta(t,\beta)\right)^{(r^\sigma)} \left(\frac{1}{2\pi}+\epsilon\eta_\sigma(t,\gamma)\right) \bigg\rbrace \d\alpha\d\beta\d\gamma \Bigg)  \Bigg] = 0,
\end{align*}
with the usual abbreviation
\begin{align*}
	\left(\frac{1}{2\pi}+\epsilon\eta(t,\alpha)\right)^{(r^\sigma)} = \prod_{i = 1}^{\abs{r^\sigma}} \left( \frac{1}{2\pi} + \epsilon \eta_{(r^\sigma_i)}(t,\alpha_i)\right).
\end{align*}
Collecting terms of order $\mathcal O(\epsilon)$ yields
\begin{align}
    \nonumber
	&\frac{\partial}{\partial t}\eta_\sigma(t,\phi) + \frac{\partial}{\partial \phi}\Bigg[ \frac{1}{(2\pi)^{2\abs{r^\sigma}+1}}  \int_\S\int_{\S^{\abs{r^\sigma}}}\int_{\S^{\abs{r^\sigma}}} g_\sigma(\alpha-\beta,\gamma-\phi)\\
	\label{eq:continuity_ordereps_pre}
	&\qquad \cdot  \left(\sum_{j = 1}^{\abs{r^\sigma}} \left( \eta_{(r^\sigma_j)}(t,\alpha_j) + \eta_{(r^\sigma_j)}(t,\beta_j)  \right)  + \eta_\sigma(t,\gamma) \right) \d\alpha\d\beta\d\gamma +\omega_\sigma \eta_\sigma(t,\phi) \Bigg] = 0.
\end{align}
Let us now interchange the sums with the integrals and evaluate each of the summands individually to obtain
\begin{align*}
	&\int_\S\int_{\S^{\abs{r^\sigma}}}\int_{\S^{\abs{r^\sigma}}} g_\sigma(\alpha-\beta,\gamma-\phi) \eta_{(r^\sigma_j)}(t,\alpha_j) \d\alpha\d\beta\d\gamma\\
	&=\int_\S\int_{\S^{\abs{r^\sigma}}}\int_{\S^{\abs{r^\sigma}}} \left( \sum_{b\in \Z^{\abs{r^\sigma}}}\sum_{l=0}^{\infty}a^\sigma_{b,l} e^{i\langle \alpha-\beta,b\rangle} e^{i(\gamma-\phi)l} + c.c.\right)\d\beta\left(\sum_{k=1}^{\infty}c_k^{(r^\sigma_j)}(t)e^{ik\alpha_j} + c.c. \right) \d\alpha\d\gamma\\
	&=(2\pi)^{\abs{r^\sigma}} \int_\S\int_{\S^{\abs{r^\sigma}}} \left( \sum_{l=0}^{\infty}a^\sigma_{\mathbf 0,l} e^{i(\gamma-\phi)l} + c.c.\right)\left(\sum_{k=1}^{\infty}c_k^{(r^\sigma_j)}(t)e^{ik\alpha_j} + c.c. \right) \d\alpha\d\gamma\\
	&=(2\pi)^{\abs{r^\sigma}} \int_\S \left( \sum_{l=0}^{\infty}a^\sigma_{\mathbf 0,l} e^{i(\gamma-\phi)l} + c.c.\right) \d\gamma\cdot  \underbrace{\int_{\S^{\abs{r^\sigma}}} \left(\sum_{k=1}^{\infty}c_k^{(r^\sigma_j)}(t)e^{ik\alpha_j} + c.c. \right) \d\alpha}_{=0} = 0.
\end{align*}
The same computations holds true if we replace $\eta_{(r^\sigma_j)}(t,\alpha_j)$ with $\eta_{(r^\sigma_j)}(t,\beta_j)$. Therefore, the sum in \eqref{eq:continuity_ordereps_pre} vanishes and as we see next, the only term that does not vanish inside the brackets of this equation is~$\eta_\sigma(t,\gamma)$. Multiplying it with the coupling function~$g_\sigma$, integrating and subsequently simplifying yields
\begin{align*}
	&\int_\S\int_{\S^{\abs{r^\sigma}}}\int_{\S^{\abs{r^\sigma}}} g_\sigma(\alpha-\beta,\gamma-\phi) \eta_\sigma(t,\gamma)\  \d\alpha\d\beta\d\gamma\\
	&=\int_{\S}\int_{\S^{\abs{r^\sigma}}}\int_{\S^{\abs{r^\sigma}}}\left( \sum_{b\in \Z^{\abs{r^\sigma}}}\sum_{l=0}^{\infty}a^\sigma_{b,l} e^{i\langle \alpha-\beta,b\rangle} e^{i(\gamma-\phi)l} + c.c.\right)\d\alpha\d\beta\left( \sum_{k=1}^{\infty}c_k^\sigma(t) e^{ik\gamma} + c.c.\right) \d\gamma\\
	&=(2\pi)^{2\abs{r^\sigma}} \int_\S \left( \sum_{l=0}^{\infty}a_{\mathbf 0,l}^\sigma e^{i(\gamma-\phi)l} + c.c.\right) \left( \sum_{k=1}^{\infty}c_k^\sigma(t) e^{ik\gamma} + c.c.\right)\d\gamma\\
	&=(2\pi)^{2\abs{r^\sigma}} \int_\S \left( \sum_{l=0}^{\infty}\bar{a}_{\mathbf 0,l}^\sigma c_l^\sigma(t) e^{-i(\gamma-\phi)l}e^{ik\gamma} + c.c.\right) \d\gamma\\
	&=(2\pi)^{2\abs{r^\sigma}+1} \left(\sum_{l=1}^{\infty}\bar{a}_{\mathbf 0,l}^\sigma c_l^\sigma(t) e^{i\phi l} + c.c.\right).
\end{align*}
Combining these two results with \eqref{eq:continuity_ordereps_pre} we get
\begin{align*}
	\frac{\partial}{\partial t} \left( \sum_{k=1}^{\infty}c_k^\sigma(t) + c.c.\right) + \frac{\partial}{\partial \phi}\left[ \left(\sum_{k=1}^{\infty}\bar{a}^\sigma_{\mathbf 0,k}c_k^\sigma(t) e^{i\phi k} + c.c.\right) + \omega_\sigma \left(\sum_{k=1}^{\infty}c_k^\sigma(t) e^{i\phi k} + c.c.\right)\right] = 0.
\end{align*}
Thus, after having taken the derivative and having collected $e^{ik\phi}$-terms, it is easy to see that $c_k^\sigma(t)$ obeys the differential equation
\begin{align}\label{eq:c_differential}
	{c_k^\sigma}'(t) = -(\bar{a}_{\mathbf 0,k}^\sigma + \omega_\sigma)ikc_k^\sigma(t).
\end{align}
Therefore, small perturbations of population $\sigma$ in direction of $e^{ik\phi} + c.c.$ with $k\ge 1$ decay on a linear level if $\Re(-(\bar{a}_{\mathbf 0,k}^\sigma + \omega_\sigma)ik) = -k\Im(a_{\mathbf 0, k}^\sigma) < 0$. Similarly, they grow if $-k\Im(a_{\mathbf 0,k}^\sigma) > 0$.
However, it is important to note that the equations \eqref{eq:c_differential} are only based on a formal derivation. Assuming nonetheless that our formal calculations can be made rigorous in this way, we can summarize our results by claiming linear stability of the all-splay state if $\Im(a_{\mathbf 0,k}^\sigma) > 0$ for all $\sigma\in[M], k=1,2,\dots$ and linear instability if $\Im(a_{\mathbf 0,k}^\sigma) < 0$ for one $\sigma \in [M]$ and one $k=1,2,\dots$.

\begin{remark}
There are several challenges when trying to obtain rigorous (linear) stability results. First, we have to rigorously linearize by constructing a suitable function space in which the operator $F$ defined by 
    \begin{align*}
        F_\sigma[\rho](\phi) = -\frac{\partial}{\partial \phi}[\rho_\sigma(\phi) V_\sigma[\rho](\phi)]
    \end{align*}
    is Fr\'{e}chet differentiable. Then, we have to check that the formal calculation above holds within this function space, and that we have described the spectrum completely. Finally, one has to invoke a suitable result that linear stability entails local nonlinear stability. Carrying out this full stability analysis program is beyond the scope of the current work.
\end{remark}

\section{Mean-Field Dynamics of Phase Oscillator Networks}\label{sec:Examples}
In this section, we give several examples to illustrate the theory and results we have developed so far.

\subsection{The Kuramoto Model for Identical Oscillators}

In the easiest case, $M=1$ and $s^1 = \{\}$, so the function $G_1$ in \eqref{eq:coupling} is mapping only from $\S$ to~$\R$. The measure $\mu_1(t)$ then simply gets transported along the time-independent velocity field
\begin{align*}
	(\mathcal K_1 \mu)(\phi) = \omega_1 + G_1(\phi).
\end{align*}
However, this case is not really interesting, which is why we now consider the case $M=1, s = s^1 = (1), G_1(\alpha,\phi) = \sin(\alpha-\phi)$. For an initial measure $\mu^\mathrm{in} = \mu^\mathrm{in}_1 \in\mathcal P(\S)$, the characteristic system \eqref{eq:SystemOfEquations}--\eqref{eq:coupling} simplifies to
\begin{subequations}
	\begin{align*}
		\partial_t \Phi(t,\xi^\mathrm{in}, \mu^\mathrm{in}) &= (\mathcal K\mu(t))(\Phi(t,\xi^\mathrm{in}, \mu^\mathrm{in}))\\
		\mu(t) &= \Phi(t,\cdot, \mu^\mathrm{in})\#\mu^\mathrm{in}\\
		\Phi(0,\xi^\mathrm{in},\mu^\mathrm{in})&= \xi^\mathrm{in},
	\end{align*}
\end{subequations}
with the coupling function
\begin{align*}
	(\mathcal K \mu)(\phi) = \omega + \int_{\S} \sin(\alpha-\phi) \ \d\mu(\alpha),
\end{align*}
where $\omega\in\R$ is the common oscillator frequency of the single population. So we have recovered the classical characteristic system of the Kuramoto model for identical oscillators. Indeed, if one assumes that the initial measure has the form of an empirical measure given by 
\begin{align*}
    \mu^\mathrm{in} = \frac{1}{N}\sum_{k=1}^{N}\delta_{\phi_k^\mathrm{in}},
\end{align*}
for some $N\in \N$, then it is easy to see that the solution $\mu(t)$ is of the form $\mu(t) = \frac{1}{N}\sum_{k=1}^{N}\delta_{\phi_k(t)}$ for functions $\phi_k(t)$ satisfying
\begin{align}\label{eq:Kuraclassic}
	\dot\phi_k(t) = \omega + \frac{1}{N}\sum_{j=1}^{N}\sin(\phi_j(t)-\phi_k(t)).
\end{align}
The equations~\eqref{eq:Kuraclassic} are the classical finite-dimensional Kuramoto model~\cite{Kuramoto1975}. In fact, one may prove that as $N\rightarrow\infty$, then the evolution of the empirical measures due to~\eqref{eq:Kuraclassic} is well-approximated by a mean-field limit~\cite{Neunzert1978} as described in Section \ref{sec:existence_uniqueness_contdepini}. 

This model can be extended by replacing the sinusoidal coupling function $G_1$ by a more general coupling function $G_1(\alpha,\phi) = f(\alpha-\phi)$. The velocity field in this generalized Kuramoto model for identical oscillators is then given by
\begin{align*}
	(\mathcal K \mu)(\phi) = \omega + \int_\S f(\gamma-\phi)\ \d\mu(\gamma).
\end{align*}
Theorem \ref{thm:sync_stability} yields the stability of the synchronized state in this system if $f\in C^1(\S)$ and $f'(0)>0$. By Theorem \ref{thm:sync_asymp_stability}, the synchronized state is asymptotically stable in the space of densities if furthermore the function $\psi\mapsto f(\psi)-f(0)$ has only one root with non-vanishing derivative around the circle except~$0$ and~$f'$ is Lipschitz continuous.

\subsection{One Population with Higher-Order Interactions}

Let us now consider a system which still consists of only one population but involves higher-order interactions. Specifically, we reconsider system \eqref{eq:System_Skardal}:
\begin{align}\label{eq:System_Skardal_reconsider}
	\dot\theta_i = \omega + \frac{K_1}{N}\sum_{j=1}^{N}\sin(\theta_j-\theta_i) + \frac{K_2}{N^2}\sum_{j=1}^{N}\sum_{l=1}^{N}\sin(2\theta_j-\theta_l-\theta_i) + \frac{K_3}{N^3}\sum_{j=1}^{N}\sum_{l=1}^{N}\sum_{m=1}^{N}\sin(\theta_j-\theta_l+\theta_m-\theta_i)
\end{align}
Here, $N$ denotes the amount of discrete oscillators, $i\in [N]$ and $K_1,K_2,K_3\in \R$. Putting this system into our framework \eqref{eq:SystemOfEquations}--\eqref{eq:coupling}, the coupling function $G := G_1\colon \S^3\times\S\to \R$ is given by
\begin{align*}
	G(\alpha,\phi) = K_1\sin(\alpha_1-\phi) + K_2\sin(2\alpha_1-\alpha_2-\phi) + K_3\sin(\alpha_1-\alpha_2+\alpha_3-\phi)
\end{align*}
and the multi-index $s^1$ is trivially given by $s^1=(1,1,1)$. Consequently, by Theorem \ref{thm:ExistenceMeasure}, there even exists a measure valued solution of the mean-field limit of the system \eqref{eq:System_Skardal_reconsider}. However, to put this system into the more restrictive form of \eqref{eq:coupling_differences}, we need to assume $K_2=0$, as we have also assumed in Section \ref{sec:Sinusoid}. In this case, the function $g:=g_1\colon \S\times\S\to\R$ from \eqref{eq:coupling_differences} reads as
\begin{align}\label{eq:System_Skardal_coupling}
	g(\alpha,\gamma) = K_1\sin(\gamma) + K_3\sin(\alpha+\gamma),
\end{align}
with multi-index $r^1 = (1)$. Note that $g(0,\gamma) = (K_1+K_3)\sin(\gamma)$. Theorem \ref{thm:sync_stability} thus tells us that the synchronized state in the system \eqref{eq:SystemOfEquations},\eqref{eq:coupling_differences},\eqref{eq:System_Skardal_coupling} is stable if $K_1+K_3>0$. Furthermore, in this case, by Theorem \ref{thm:sync_asymp_stability}, it is even asymptotically stable in the space of densities.

\begin{example}[Example \ref{exam:WatanabeStrogatzExample} revisited]
Again, let $K_1=1$ and $K_3 = -4$. Then, $K_1 + K_3 = -3$ and thus, the synchronized state is not stable. This is consistent with Example \ref{exam:WatanabeStrogatzExample}, since there we found $r=1/2$ to be attractive. In particular, almost all initial conditions globally converge to $r=1/2$ and so it makes sense that the synchronized state is not stable.
\end{example}

Note that the network interactions in~\eqref{eq:System_Skardal_reconsider} are quite specific: The coupling functions are purely sinusoidal and thus the authors use the Ott--Antonsen reduction to understand the system dynamics~\cite{SkardalArenas}. Note that our analysis here is not limited to such networks but applies also to phase oscillator networks with general (higher-order) interactions where the reduction methods cease to apply. These networks arise naturally, for example, through phase reductions in oscillator networks that generically contain multiple harmonics; cf.~\cite{Ashwin2016a,Bick2016b,Leon2019a}.

\subsection{Multiple Coupled Populations with Higher-Order Interactions}\label{sec:ExampleBick}

In \cite{Bick2019a} networks of $M=3$ finite phase oscillator populations coupled by higher-order interactions have been considered. In particular, the main equations from~\cite{Bick2019a} are given by
\begin{align}\label{eq:bick_equation_1}
    \dot \phi_{\sigma,k} = \omega + \sum_{\substack{j=1\\j\neq k}}^N \Big( h_2(\phi_{\sigma,j}-\phi_{\sigma,k}) - K^- H_4(\phi_{\sigma-1}; \phi_{\sigma,j}-\phi_{\sigma,k}) + K^+ H_4(\phi_{\sigma+1}; \phi_{\sigma,j}-\phi_{\sigma,k})\Big),
\end{align}
where the index $\sigma\pm 1$ for the population has to be understood modulo $M$ if $\sigma\pm 1\not\in\{1,2,3\}$. Furthermore, $\phi_{\sigma,k}$ refers to the phase of the $k$th oscillator in population $\sigma$ for $k = 1,\dots,N$, $h_2\colon \S\to \R$ is a Lipschitz-continuous intra-population coupling function and
\begin{align}\label{eq:bick_equation_2}
    H_4(\phi_\tau; \phi) = \frac{1}{N^2} \sum_{n,m=1}^N h_4(\phi_{\tau,m}-\phi_{\tau,n} + \phi)
\end{align}
with a Lipschitz-continuous inter-population coupling function $h_4\colon \S\to \R$.
It can be seen that if one population, say the first, is initially synchronized, i.e., $\phi_{1,1}(0)=\dots=\phi_{1,N}(0)$ then it is also synchronized at later times. Similarly, if the the oscillators of the first population are initially in splay state, they exhibit this property also at later times. 
If two populations are fixed to be either synchronized or in splay state, the oscillators in the remaining free population $\hat\sigma$ evolve according to 
\begin{align*}
    \dot \phi_{\hat\sigma,k}(t) = \omega + \frac 1N \sum_{j=1}^N g(\phi_{\hat\sigma,j}(t)-\phi_{\hat\sigma,k}(t)),\qquad k=1,\dots,N,
\end{align*}
where the coupling function $g$ is made up of $h_2$ and $h_4$ and depends on the state in which the fixed populations are locked. 
More interestingly, by choosing appropriate coupling functions $h_2,h_4$ and coupling constants $K^+, K^-$, one can achieve that the (de)synchronization of one population can cause another population to start synchronizing or desynchronizing. In particular, existence of a heteroclinic cycle for small populations was shown in~\cite{Bick2019a}. Within these heteroclinic cycle populations alternatingly synchronize and desynchronize. Our numerical simulations have confirmed that such a cycle continues to exist if the size of the populations grows to infinity.

With little change to the system \eqref{eq:bick_equation_1},\eqref{eq:bick_equation_2}, this system can be written in the form of \eqref{eq:SystemOfEquations},\eqref{eq:coupling}. In fact, when summing over all $j$ from $1$ to $N$ in \eqref{eq:bick_equation_1} instead of over all $j$ except $j=k$, the coupling in this system is of the form \eqref{eq:coupling} when $s^1 = (3,3,2,2,1), s^2 = (1,1,3,3,2), s^3=(2,2,1,1,3)$ and the coupling functions are defined by $G_1(\alpha,\phi) = G_2(\alpha,\phi) = G_3(\alpha,\phi) := G(\alpha,\phi)$ with
\begin{align*}
    G(\alpha,\phi) = h_2 (\alpha_5-\phi) - K^-h_4(\alpha_1-\alpha_2, \alpha_5-\phi) + K^+h_4(\alpha_3-\alpha_4,\alpha_5-\phi).
\end{align*}
The velocity field \eqref{eq:coupling} then evaluates to
\begin{align}
	\nonumber
	(\mathcal K_\sigma\mu)(\phi) = \omega_\sigma + &\int_\S \Big[ h_2(\gamma-\phi)\\
	\nonumber
	& - K^-\int_\S\int_\S h_4(\alpha-\beta,\gamma-\phi)\ \mu_{\sigma-1}(\d \alpha) \mu_{\sigma-1}(\d \beta)\\
	\nonumber
	& + K^+\int_\S\int_\S h_4(\alpha-\beta,\gamma-\phi)\ \mu_{\sigma+1}(\d \alpha) \mu_{\sigma+1}(\d \beta) \Big] \ \mu_\sigma(\d \gamma).
 \end{align}
Our results allow to analyze the stability of invariant sets of these networks in the mean-field limit.

Note that all of the multi-indices are of the special form \eqref{eq:multiindex_specialform}. Thus, with $r^1 = (3,1), r^2 = (1,3), r^3 = (2,1)$ this system can also be put into the form \eqref{eq:coupling_differences}. Then, the coupling functions are given by $g_1(\alpha,\phi) = g_2(\alpha,\phi) = g_3(\alpha,\phi) := g(\alpha,\phi)$, with
\begin{align*}
	g(\alpha,\gamma) = h_2(\gamma) - K^-h_4(\alpha_1, \gamma) + K^+h_4(\alpha_2,\gamma).
\end{align*}
Having put the system into the form \eqref{eq:coupling_differences} allows us to apply results from Section \ref{sec:dynamics}. For example Theorem~\ref{thm:sync_stability} yields the stability of the all-synchronized state if the function
\begin{align*}
	f(\gamma) := g(0,\gamma) = h_2(\gamma) + (K^+-K^-) h_4(0,\gamma)
\end{align*}
satisfies $f'(0)>0$.

Let us now try to investigate the stability of $\Sp\Dp\Dp$ 
 with respect to perturbations in the first population. Unfortunately, we cannot apply Theorem \ref{thm:sync_stability} immediately but we have to do some preparatory steps first. So we first choose $\mu_2^\mathrm{in}(A) = \mu_3^\mathrm{in}(A) = \frac{1}{2\pi}\lambda_\S(A)$. Then, Proposition \ref{prop:invariant_subspaces} causes the second and third population to stay in splay state for all $t\ge 0$. The velocity field according to which $\mu_1(t)$ is transported is given by
\begin{align*}
	(\mathcal K_1 \mu(t))(\phi) = \omega_1 + \int_\S \left[ h_2(\gamma-\phi) + \frac{K^+-K^-}{2\pi}\int_\S h_4(\alpha,\gamma-\phi)\d\alpha \right] \ \mu_1(\d\gamma).
\end{align*}
Therefore, the dynamics in $\mu_1\Dp\Dp$ can be described by a single coupling function
\begin{align*}
	\hat g(\gamma) := h_2(\gamma) + \frac{K^+-K^-}{2\pi}\int_\S h_4(\alpha,\gamma)\d\alpha.
\end{align*}
Only now, we can apply Theorem \ref{thm:sync_stability} to see that $\Sp\Dp\Dp$ is stable with respect to perturbations in the first population if the function $\hat g$ satisfies $\hat g'(0)>0$.

In order to investigate the linear stability of~$\Dp\Dp\Dp$ we need to assume Fourier expansions
\begin{align*}
	h_2(\gamma) &= \sum_{k=1}^{\infty}\xi_k e^{i\gamma k} + c.c.,\\
	h_4(\alpha,\gamma) &= \sum_{l=-\infty}^{\infty}\sum_{k=0}^{\infty} \zeta_{l,k} e^{i\alpha l} e^{i\gamma k} + c.c.
\end{align*}
By inserting them into the representation of the coupling function~$g$, we see that
\begin{align*}
	g(\alpha_1,\alpha_2,\phi) = \sum_{l_1=-\infty}^{\infty}\sum_{l_2=-\infty}^{\infty}\sum_{k=0}^{\infty}a_{l,k} e^{i\alpha_1 l_1}e^{i\alpha_2 l_2} e^{i\phi k} + c.c.
\end{align*}
for Fourier coefficients~$a_{l,k}$ that satisfy $a_{\mathbf 0, k} = \xi_k - K^-\zeta_{0,k} + K^+\zeta_{0,k}$. By the results obtained in Section~\ref{sec:SplayStability}, the all-splay state is linearly stable if
\begin{align*}
	\Im(\xi_k + (K^+-K^-)\zeta_{0,k})>0
\end{align*}
for all $k = 1,2,\dots$ and linearly unstable if one of these coefficients is less than~$0$.

These results give necessary conditions for the emergence of heteroclinic cycles involving the invariant sets~$\Sp\Dp\Dp, \Sp\Sp\Dp, \dotsc$ to exist not only in networks of finitely many oscillators but also in the mean-field limit of these systems. Note that a similar analysis is possible for the mean-field limit of networks that consist of $M=4$ coupled oscillator populations that support heteroclinic networks with multiple cycles in~\cite{Bick2019}.

\section{Discussion and Outlook}\label{sec:conclusion}

In this paper, we have proposed a new general framework for the evolution of coupled phase oscillator populations with higher-order coupling. First, we provided a general solution theory. We clarified existence and uniqueness of weakly continuous solutions to the characteristic system and justified the mean-field limit. Then, we studied some dynamical properties of the characteristic system. We showed that the subspaces, which are characterized by one or more populations being either in synchronized or in splay state, are dynamically invariant. Next, we proved the stability of the all-synchronized state under concrete conditions on the coupling functions. Finally, we also analyzed developed linear stability analysis for the splay state via the mean-field limit equation. 

Although we have provided the general mathematical foundations for studying large-scale multi-population oscillator networks with higher-order coupling, there are still many open questions for future work. Here, we considered the case of populations with identical oscillators. This means that in the mean-field limit there are atomic measures that are invariant under the flow. There are two ways to break this degeneracy. First, one can assume that the intrinsic frequencies of the oscillators follow a distribution with a density as in the classical Kuramoto model; cf.~\cite{Bick2018c}. Second, adding noise to the evolution leads to a diffusive terms in the Fokker--Planck equation~\cite{Tyulkina2018}. In either case, the synchronized phase configuration is not invariant anymore and deforms to a near-synchronous stationary solution. Insights into how the stability properties derived here change through these perturbations would be desirable. Unlike the synchronized state, the splay state would stay invariant for non-identical oscillators but investigating its linear stability would be more complicated due to the frequency dependence. One expects bifurcation structures to be affected generically by higher-order coupling, e.g., being able to change super- to sub-critical transitions~\cite{KuehnBick,SkardalArenas}.

If the network interactions contain just a single harmonic, the Watanabe--Strogatz reduction applies. Its application in the mean-field limit has so far been heuristic and one typically assumes the existence of densities~\cite{Pikovsky2011} and it would be interesting to understand Watanabe--Strogatz theory for the characteristic equation~\eqref{eq:SystemOfEquations} that describes the evolution of general measures. Together with nonidentical frequencies within a populations, this would be the first step towards a rigorous description of Ott--Antonsen theory in a measure theoretic sense; see also~\cite{Engelbrecht2020}.

While we discussed network dynamical systems with higher-order interactions, we did not assign an explicit algebraic structure to the dynamical equations. Recently, dynamical systems on higher-order networks---whether hypergraphs or simplicial complexes---have attracted attention. While the assignment of higher-order network structure may not be unique, this perspective has its advantages: It naturally leads to limiting systems involving hypergraph variants of graphons~\cite{Medvedev2014a,Medvedev2014,KuehnThrom2}, or more generally hypergraph variants of graphops~\cite{Backhausz2018,KuehnGraphops,GkogkasKuehn}. In this context, one could also aim to link the stability analysis via the Vlasov--Fokker--Planck equation for hypergraphs better with direct methods on the level of the finite-dimensional ODEs for hypergraphs such as master stability functions~\cite{MulasKuehnJost,Battistonetal1}.\medskip

\textbf{Acknowledgements:} We thank an anonymous referee for helpful suggestions, which have helped to improve the paper. The authors gratefully acknowledge the support of the Institute for Advanced Study at the Technical University of Munich through a Hans Fischer Fellowship awarded to CB that made this work possible. CB acknowledges support from the Engineering and Physical Sciences Research Council (EPSRC) through the grant EP/T013613/1. CK acknowledges support via a Lichtenberg Professorship.

\bibliographystyle{plain}
\bibliography{OscillatorPopulations}

\appendix
\section{Derivation of higher-order systems}
In this section, we derive the form of higher-order interactions following \cite{Ashwin2016a}, which provides, beyond the many practical applications of higher-order coupling cited in Section~\ref{sec:introduction}, a clear mathematical motivation for the class of model we consider. The starting point of the derivation in~\cite{Ashwin2016a} is a system of identical particles interacting with each other via the following ordinary differential equations:
\begin{align}\label{eq:derivation_main_system}
\begin{split}
    \frac{\mathrm d}{\mathrm dt}x_1 &= H_\lambda(x_1) + \epsilon h_{\lambda, \epsilon}(x_1; x_2,\dots,x_N),\\
    \vdots \quad & \vdots \quad \vdots\\
    \frac{\mathrm d}{\mathrm dt}x_N &= H_\lambda(x_N) + \epsilon h_{\lambda, \epsilon}(x_N; x_1,\dots,x_{N-1}),
\end{split}
\end{align}
where $x_k \in \R^d$ with $d\ge 2$ for all $k = 1,\dots,N$, $\epsilon\in \R$ is a coupling parameter and $\lambda$ is a bifurcation parameter such that each system undergoes a Hopf bifurcation at $\lambda = 0$ when $\epsilon = 0$. Furthermore, assume that the uncoupled system 
\begin{align}\label{eq:derivation_uncoupled}
    \frac{\mathrm d}{\mathrm dt} x = H_\lambda(x)
\end{align}
has a linearly stable fixed point for $\lambda<0$ and undergoes a Hopf bifurcation at $\lambda = 0$. The Jacobian matrix $DH_\lambda(0)$ is assumed to have a complex pair of eigenvalues $\lambda \pm \textnormal{i} \omega$ with $\omega \neq 0$ and all other eigenvalues of $DH_\lambda(0)$ are supposed to have negative real part.
Moreover, the coupled system \eqref{eq:derivation_main_system} is assumed that $(x_1,\dots,x_N)=0$ is an equilibrium for all $(\lambda,\epsilon)$ in a neighborhood of $(0,0)$.

Using equivariant bifurcation theory, the authors of \cite{Ashwin2016a} reduced the system \eqref{eq:derivation_main_system} on a center manifold, described by $(z_1,\dots,z_N)\in \C^N$, such that $z_k$ reflects the center manifold of $x_k$ when $\lambda = \epsilon = 0$. This system is given by
\begin{align}\label{eq:derivation_center_manifold}
\begin{split}
    \frac{\mathrm d}{\mathrm dt}z_1 &= f_\lambda(z_1) + \epsilon g_\lambda(z_1; z_2,\dots,z_N) + \mathcal O(\epsilon^2),\\
    \vdots\quad&\vdots \quad\vdots\\
    \frac{\mathrm d}{\mathrm dt}z_N &= f_\lambda(z_N) + \epsilon g_\lambda(z_N; z_2,\dots,z_{N-1}) + \mathcal O(\epsilon^2).
\end{split}
\end{align}
Since the uncoupled system \eqref{eq:derivation_uncoupled} has a supercritical hopf bifurcation at $\lambda = 0$ and the equilibrium point at the origin is stable for $\lambda < 0$, we expect the coupled system to have an invariant attracting torus whenever $\lambda>0$ and $\epsilon$ are close to $0$. In fact, a theorem from \cite{Ashwin2016a} shows exactly that and furthermore states that the flow on the invariant torus can be approximated by a higher-order coupled oscillator system.

\begin{theorem}[{\cite[Theorem 3.2]{Ashwin2016a}}] Consider system \eqref{eq:derivation_center_manifold} with $S_N$-symmetry (for fixed $N$) such that the $N$ uncoupled systems $(\epsilon = 0)$ undergo a generic supercritical Hopf bifurcation on $\lambda$ passing through $0$. There exists $\lambda_0>0$ and $\epsilon_0 = \epsilon_0(\lambda)$ such that for any $\lambda\in (0,\lambda_0)$ and $|\epsilon| < \epsilon_0(\lambda)$ the system \eqref{eq:derivation_center_manifold} has an attracting $C^r$-smooth invariant $N$-dimensional torus for arbitrarily large $r$. Moreover, on this invariant torus, the phases $\phi_j$ of the flow can be expressed as a coupled oscillator system
\begin{align}\label{eq:derivation_higher_order_approximation}
\begin{split}
    \frac{\mathrm d}{\mathrm dt}\phi_j  &=\tilde \Omega(\phi, \epsilon) +  \frac{\epsilon}{N}\sum_{k=1}^N g_2(\phi_k-\phi_j) + \frac{\epsilon}{N^2}\sum_{k,l=1}^N g_3(\phi_k+\phi_l-2\phi_j)\\
    &\quad + \frac{\epsilon}{N^2} \sum_{k,l=1}^N g_4(2\phi_k-\phi_l-\phi_j) + \frac{\epsilon}{N^3}\sum_{k,l,m=1}^N g_5(\phi_k+\phi_l-\phi_m-\phi_j)\\
    &\quad + \epsilon \tilde g_j(\phi_1,\dots,\phi_N) + \mathcal O(\epsilon^2)
\end{split}
\end{align}
for fixed $0<\lambda<\lambda_0$ in the limit $\epsilon\to 0$, where $\tilde \Omega(\phi,\epsilon)$ is independent of $j$ and
\begin{align*}
    g_2(phi) &= \xi^0_1\cos(\phi + \chi_1^1) + \lambda \xi^1_1\cos(\phi + \chi_1^1) + \lambda \xi_2^1\cos(2\phi + \chi_2^1),\\
    g_3(\phi) &= \lambda \xi_3^1 \cos(\phi + \chi_3^1),\\
    g_4(\phi) &= \lambda \xi_4^1 \cos(\phi + \chi_4^1),\\
    g_5(\phi) &= \lambda \xi_5^1 \cos(\phi + \chi_5^1).
\end{align*}
The constants $\xi_i^j$ and $\chi_i^j$ are generically non-zero. The error term satisfies
\begin{align*}
    \tilde g(\phi_1,\dots,\phi_N) = \mathcal O(\lambda^2)
\end{align*}
uniformly in the phases $\phi_k$. The truncation of \eqref{eq:derivation_higher_order_approximation} by removing $\tilde g$ and $\mathcal O(\epsilon^2)$ terms is valid over time intervals $0<t<\tilde t$ where $\tilde t = \mathcal O(\epsilon^{-1}\lambda^{-2})$ in the limit $0<\epsilon\ll \lambda \ll 1$. In particular, for any $N$, this approximation involves up to four interacting phases.
\end{theorem}

Specifically, as stated in \cite{Ashwin2016a}, over a longer timescale a system of the form
\begin{align}
\begin{split}
    \frac{\mathrm d}{\mathrm dt}\phi_j  &=\tilde \Omega(\phi, \epsilon) +  \frac{\epsilon}{N}\sum_{k=1}^N g_2(\phi_k-\phi_j) + \frac{\epsilon}{N^2}\sum_{k,l=1}^N g_3(\phi_k+\phi_l-2\phi_j)\\
    &\quad + \frac{\epsilon}{N^2} \sum_{k,l=1}^N g_4(2\phi_k-\phi_l-\phi_j) + \frac{\epsilon}{N^3}\sum_{k,l,m=1}^N g_5(\phi_k+\phi_l-\phi_m-\phi_j)\\
\end{split}
\end{align}
better approximates the dynamics on the invariant torus than a system without higher-order coupling.

\end{document}